\documentclass[11pt]{article}
\usepackage[left=3cm,right=3cm,top=4cm,bottom=4.5cm]{geometry}

\usepackage{amsmath,amsthm,amssymb}
\usepackage{mathtools}
\usepackage{tikz-cd}

\usepackage[T1]{fontenc}     

\usepackage{enumerate}

\usepackage{soul}

\usepackage{fourier}

\newtheorem{thm}{Theorem}
\newtheorem{prop}[thm]{Proposition}
\newtheorem{lem}[thm]{Lemma}
\newtheorem{cor}[thm]{Corollary}

\newtheorem{Question}[thm]{Question}

\theoremstyle{remark}
\newtheorem{rem}[thm]{\textbf{Remark}}
\newtheorem{cl}[thm]{\textbf{Claim}}

\theoremstyle{definition}

\newtheorem{defn}[thm]{Definition}

\newcommand{\R}{\mathbb{R}}

\newcommand{\N}{\mathbb{N}}
\newcommand{\W}{\mathcal{W}}
\newcommand{\Diff}{\mathit{Diff}}

\DeclareMathOperator{\diam}{diam}
\DeclareMathOperator{\diff}{diff}

\AtEndDocument{%
  \bigskip
  {\footnotesize
  \textsc{A. Algom, Department of Mathematics, The University of Haifa at Oranim, 36006 Tivon, Israel.}
  \textit{E-mail address:} \texttt{amir.algom@math.haifa.ac.il.}\par
  }%

  \bigskip
  {\footnotesize
  \textsc{S. Ben Ovadia, Einstein Institute of Mathematics, The Hebrew University of Jerusalem, 91904 Jerusalem, Israel.}
  \textit{E-mail address:} \texttt{Snir.BenOvadia@mail.huji.ac.il.}\par
  }%

  \bigskip
  {\footnotesize
  \textsc{F. Rodriguez Hertz, Department of Mathematics, The Pennsylvania State University, University Park, PA 16802, USA.}
  \textit{E-mail address:} \texttt{hertz@math.psu.edu.}\par
  }%

  \bigskip
  {\footnotesize
  \textsc{M. Shannon, Facultad de Ingenier\'ia, Universidad de la Rep\'ublica, Julio Herrera y Reissig 565, 11300 Montevideo, Uruguay.}
  \textit{E-mail address:} \texttt{shannon@cmat.edu.uy.}\par
  }%
}%

\begin{document}

\title{How linear can a non-linear hyperbolic IFS be?}
\author{A. Algom, S. Ben Ovadia, F. Rodriguez Hertz, and M. Shannon}
\date{}

\maketitle

\begin{abstract}
Motivated by a question of M. Hochman, we construct examples of hyperbolic IFSs $\Phi$ on $[0,1]$ where linear and non-linear behaviour coexist. Namely, for every $2\leq r \leq \infty$ we exhibit the existence of a $C^r$-smooth IFS such that $f'\equiv c(\Phi)$ on the attractor and $f''\equiv 0$ for every $f \in \Phi$, yet $\Phi$ is not $C^t$-smooth for any $t>r$, nor $C^r$-conjugate to self-similar. We provide a complete classification of these systems. Furthermore, when $r>1$, we give a necessary and sufficient Livsic-like matching condition for a self-conformal $C^r$-smooth IFS to be conjugated to one of these systems having $f''=0$ on the attractor, for every $f\in \Phi$. We also show that this condition fails to ensure the existence of a $C^1$-conjugacy in mere $C^1$-regularity.
\end{abstract}


\section{Introduction} \label{Section: Intro}
Let $\Phi = \lbrace f_1,...,f_n \rbrace\subseteq C^r ([0,1])$ be a finite collection of non-singular strict contractions, where $1\leq r\leq \infty$ or $r=\omega$. The action on $[0,1]$ generated by arbitrary composition of the elements in $\Phi$ is called a \textit{self-conformal IFS} (Iterated Function System). It is well known that there exists a unique compact non-empty subset $X=X_\Phi\subseteq [0,1]$, called the attractor of $\Phi$, that satisfies
$$X=\bigcup_{i=1} ^n f_i (X).$$
We say that $\Phi$ is $C^r$-conjugated to another IFS $\widetilde{\Phi}$ if  there exists a $C^r$ diffeomorphism $h$ on $[0,1]$ such that
$$\widetilde{\Phi} = \lbrace h\circ f_i \circ h^{-1} \rbrace_{i=1} ^n.$$
Finally, $\Phi$ is called \textit{self-similar} if $\Phi$ is fully made up of affine maps. See e.g. Bishop-Peres \cite{bishop2013fractal} for a standard introduction, and B\'ar\'any-Simon-Solomyak \cite[Chapter 14]{Barany2023Simon} for a more recent account.

We are interested in the following question of Hochman\footnote{This problem was posed to the first named author when he was a graduate student at the Hebrew University of Jerusalem in the Winter of 2014.}.
\begin{Question} \label{Question_Hochman}
Let $2\leq r \leq \infty$ and let $\Phi \subseteq C^r([0,1])$ be a self-conformal IFS. Is it possible that $\Phi$ is \emph{linear}, in the sense that  $f'' (x) = 0$ for all $x\in X_\Phi$ and $f\in \Phi$, yet $\Phi$ is not $C^r$ conjugate to a self-similar IFS? 
\end{Question}
Question \ref{Question_Hochman} is motivated by   methods Hochman devised with Shmerkin to study normal numbers in fractals \cite[see comments made during the proof of Theorem 1.5]{hochmanshmerkin2015}; By the study of self-symmetries of self-similar sets by Elekes-Keleti-M{\'a}th{\'e} \cite{elekes2010self}; By the work of Bedford-Fisher \cite{BedfordFisher1997Ratio} about upgrading the smoothness of conjugating maps between hyperbolic Cantor sets (to be defined soon), which is closely related with the work of Sullivan \cite{Sullivan1987ratio} and his scaling function, among others. We will come back to this last point below. Our interest in this Problem was rekindled by recent progress on the study of Fourier dimension of self-conformal sets. More specifically, the very same linearity condition as in Hochman's Question \ref{Question_Hochman} was shown in \cite{algom2021decay} to imply that all self-conformal measures, a certain family of stationary measures on $X_\Phi$, have logarithmic Fourier decay.  In fact, for $C^2 (\mathbb{R})$ smooth IFSs it was later shown that being $C^2$-conjugated to linear (recall our definition of linear from Question \ref{Question_Hochman}) is the \textit{only} possible obstruction for $X_\Phi$ to having positive Fourier dimension \cite{algom2023polynomial, baker2023spectral}. This essentially resolved the Fourier decay problem for smooth non-linear IFSs. We will return to this issue and the implication our analysis has on the Fourier decay problem after we state and discuss our main results.

In this paper we give a full answer to Question \ref{Question_Hochman}, and study further dynamically flavoured problems that naturally arise from our analysis. Since our main concern is with the construction of the examples in Theorem \ref{Theorem A} - part b., we will restrict our analysis to the case where the IFSs have only two generators, which is enough to answer Question \ref{Question_Hochman}. We will work with maps in regularity $C^1$, regularity $C^{s,\alpha}$ for $1\leq s<\infty$ and $0<\alpha\leq 1$, regularity $C^\infty$ and regularity $C^\omega$. The index $s=1$, $s=1+\alpha$, $s=\infty$ or $s=\omega$ will be called the regularity of the map. See Section \ref{sec_preliminaries} for definitions.

To describe our results, following e.g. Bedford-Fisher \cite{BedfordFisher1997Ratio}, let us recall the definition of a hyperbolic IFS. A self-conformal IFS $\Phi=\lbrace f_0, f_1 \rbrace \subset C^{s} ([0,1])$, where $1\leq s\leq\infty$ or $s=\omega$, is called a \emph{hyperbolic IFS} if 
$0<\inf_{t\in[0,1]}\left(|f_i'(t)|\right)\leq \sup_{t\in[0,1]}\left(|f_i'(t)|\right) <1$ for both $i=0,1$ and
$$0=f_0(0)<f_0(1)<f_1(0)<f_1(1)=1.$$
Up to an affine change of coordinates, this is equivalent to the so-called  convex strong separation condition  - see \cite[Section 3]{feng2014self}. Note that this property implies that the attractor $X_\Phi$ is a Cantor set of zero Lebesgue measure. We also introduce a new notion, that will accompany us throughout the paper: A self-conformal IFS $\Phi\subseteq C^1([0,1])$ is called \textit{pseudo-affine with slope} $\lambda$ (or just pseudo-affine if we don't care about the slope) if:
\begin{equation} \label{defn_pseudo-affine_IFS}
f'(x)=\lambda \text{ for all } x\in X_\Phi \text{ and all } f\in \Phi.
\end{equation}
Note that if $\Phi\subseteq C^2([0,1])$ and is pseudo-affine then $\Phi$ is linear in the sense of Question \ref{Question_Hochman}. Indeed, this follows from pseudo-affinity and the fact that $X_\Phi$ is a perfect compact set. It is also interesting to note that the slope $\lambda$ can only take values in the set $0<\lambda<1/2$ (cf. proof of Claim  \ref{claim_1-prop_classification_2}).

We are now ready to state our first main result which answers Question \ref{Question_Hochman}:
\begin{thm} \label{Theorem A}
For every parameter $0<\lambda<1/2$ and $1\leq s\leq \infty$, there exist hyperbolic pseudo-affine IFSs $\Phi \subseteq C^s([0,1])$ with slope $\lambda$, satisfying that $\Phi \not \subseteq C^t([0,1])$ for every $s<t$. In the case $s=\infty$, this means that there exist such IFSs that are $C^\infty$ but not real analytic. Moreover, in each degree of regularity $1\leq s\leq\infty$, the following is true:
\begin{enumerate}[a.]
\item There exist examples of these IFSs that are $C^s$-conjugated to self-similar;
\item There exist examples of these IFSs that are not $C^s$-conjugated to self-similar.
\end{enumerate}
\end{thm}
The examples found in Theorem \ref{Theorem A} - part b. answer in full Hochman's Question \ref{Question_Hochman}. Let's explain another question coming from the works of Bedford-Fisher and Sullivan, that is related to Theorem \ref{Theorem A} - part a. We refer to Section \ref{sec_preliminaries} for an account on the concepts used here. The invariant Cantor set of a hyperbolic IFS is called a \emph{hyperbolic Cantor set}. Determining whether two hyperbolic IFSs $\Phi$ and $\widetilde{\Phi}$ of class $C^s$ are $C^s$-conjugated is equivalent to determine whether the Cantor sets $X_\Phi$ and $X_{\widetilde{\Phi}}$ are $C^s$-diffeomorphic, under an orientation preserving diffeomorphism of the interval. This can be completely understood in the case of regularity $s=1+\alpha$ with $0<\alpha\leq 1$, due to an invariant called \emph{scaling function}. The fact that these invariant Cantor sets $X$ are hyperbolic attractors of systems of class $C^{1,\alpha}$ imply that their fine scale structure can be encoded in a $\alpha$-{H}\"older function $\rho:X^*\to\Delta$, called the \emph{scaling function of $X$}, with domain in the \emph{dual Cantor set of $X$} and range in the simplex $\Delta=\{(t_1,t_2,t_3)\in\R^3:0\leq t_i\leq 1,\ t_1+t_2+t_3=1\}$. It turns that \emph{the scaling functions are a complete invariant of $C^{1,\alpha}$-conjugacy classes of hyperbolic Cantor sets of class $C^{1,\alpha}$. }
The problem of $C^s$-classification becomes more subtle when we increase the degree of regularity to $s\geq 2$. Since the Cantor sets $X_\Phi$ and $X_{\widetilde{\Phi}}$ are in particular of class $C^{1+\alpha}$ then they have well-defined scaling functions $\rho_\Phi$ and $\rho_{\widetilde{\Phi}}$, and hence they are $C^{1+\alpha}$ conjugated if and only if $\rho_\Phi=\rho_{\widetilde{\Phi}}$. In \cite{BedfordFisher1997Ratio}-\cite{Sullivan1987ratio} it is proved the following result, known as \emph{rigidity}: \emph{Two hyperbolic Cantor sets with regularity $s>1$ are $C^1$-diffeomorphic if and only if they are $C^s$-diffeomorphic.} It follows that, in regularity $s\geq 2$, the association of its scaling function to any $C^s$-hyperbolic Cantor set, is an injective association.

It is less understood, however, which are the $\alpha$-{H}\"older functions $\rho:X^*\to\Delta$ that can occur as the scaling function of a hyperbolic Cantor set $X$ of class $C^s$, for a given $s\geq 2$. More precisely, following \cite{BedfordFisher1997Ratio} and \cite{Sullivan1987ratio}, the question that arises is the following:

\begin{Question}
Let $\rho:X^*\to\Delta$ be the scaling function of a hyperbolic Cantor set $X$ of class $C^{1,\alpha}$ for some $\alpha>0$, which is an $\alpha$-{H}\"older function. Given $s\geq 2$, which properties of $\rho$ ensures that $X$ has regularity at least $s$ ? For example, if $\rho$ is constant, does it ensures that $X$ is real analytic ?
\end{Question}

A first result in this direction can be found in \cite{TP}, where it is shown that a bound on the regularity of the scaling function imposes a bound on the actual regularity of $X$. On the other hand, the examples in Theorem A - part a. show that a high degree of regularity of the scaling function (constant in this case, since these IFSs are conjugated to self-similar) does not imply any kind of regularity of the Cantor set, other than $s=1+\alpha$. More precisely, given $0<\lambda<1/2$, denote by $X_\lambda$ the self-similar Cantor set of slope $=\lambda$, which has constant scaling function $\rho_\lambda=(\lambda,1-2\lambda,\lambda)$. For every $1\leq s\leq\infty$ or $s=\omega$, let $\mathcal{E}_\lambda^s$ be the set of $C^s$-equivalence classes of hyperbolic Cantor sets of class $C^s$, having constant scaling function $=\rho_\lambda$. Then, the examples in Theorem A - part a. show that:

\begin{cor}
It is verified that $\mathcal{E}_\lambda^\omega\subsetneq\mathcal{E}_\lambda^\infty\subsetneq\mathcal{E}_\lambda^t\subsetneq\mathcal{E}_\lambda^s\subsetneq\mathcal{E}_\lambda^1$, for all $1<s<t<\infty$.
\end{cor}

Theorem \ref{Theorem A}, and our analysis in general, can also be viewed as a refinement of a result of Bam\'on,  Moreira,  Plaza, and Vera \cite{Bamon1997gogo}. In that paper, the authors find explicit conditions for the characterization and classification of a particular class of $C^{s}$-regular hyperbolic Cantor sets, called \emph{central Cantor sets}. This is done in terms of the asymptotics of various products of the sizes of gaps removed in the construction of $X_\Phi$ as a Cantor set (see Section \ref{Section: binary cantor set} for more details). In particular, examples as those in our Theorem \ref{Theorem A} - part a. have been already produced in \cite{Bamon1997gogo}; They also produced such examples where the difference set can have any chosen volume measure; And, they classified their $C^s$-conjugacy classes.

Our major new contribution compared with \cite{Bamon1997gogo} is that we provide a mechanism that ensures $\Phi$ is linear yet not conjugate to self-similar. Our construction is carried out in terms of a pair of real functions on $\lbrace 0,1\rbrace^*$, called \emph{dynamical proportions}.  These functions codify the \emph{proportions} between different \emph{gaps} in the complement of the Cantor set. The crux of the matter is that, when a Cantor set has zero Lebesgue measure, then each pair of admissible dynamical proportions uniquely determine the Cantor set as a subset of the interval, and hence all the properties of the Cantor set are encoded in this pair of functions. For instance, pseudo-affinity can be characterized in these terms, and is equivalent to the condition that the dynamical proportions are asymptotically constant with respect to the length of the word. Moreover, the deviation from the limit constant $\lambda$ encodes precise information about the regularity of the IFS.

We will prove Theorem \ref{Theorem A} by showing first a complete classification of hyperbolic pseudo-affine IFSs of class $C^s$ up to $C^s$-conjugation, for every regularity index $1\leq s\leq\infty$. This is the content of Theorem \ref{Theorem 27} in Section \ref{sec_pseudo-affine_IFS}, and it is the heart  of the construction in this paper. To avoid technicalities, we skip its statement here and refer the reader directly to that section. It is interesting to note that for determining $C^s$-conjugation in regularity $1\leq s\leq\infty$, it is enough to check $C^1$-conjugation. This is because one can invoke once again the rigidity results of Bedford-Fisher and Sullivan, and obtain that a conjugation of class $C^1$ automatically upgrades to a conjugation of class $C^s$. Observe that the problem of $C^s$-classification is trivial when $s=\omega$ (due to analyticity).

In a second approach, we study the problem of whether a general hyperbolic IFS is smoothly conjugated to a pseudo-affine IFS. Consider a hyperbolic IFS $\widetilde{\Phi}= \lbrace g_0,g_1 \rbrace\subseteq C^s([0,1])$ where $s\geq 1$. A necessary condition for $\widetilde{\Phi}$ to be $C^s$-conjugated to a pseudo-affine IFS of slope $\lambda$ is the following: 
\emph{
For every $w\in \lbrace0,1\rbrace^*$ let $G_\omega = g_{\omega_{1}}\circ \circ \circ g_{\omega_n}$ be the corresponding cylinder map, and let $y_w \in X_{\widetilde{\Phi}}$ denote its unique fixed point; That is, $G_w(y_w)=y_w$. Then: 
\begin{equation} \label{Eq Livsic condition}
\text{ There exists } 0<\lambda<1/2 \text{ such that for all } w\in \lbrace 0,1\rbrace^* \text{ we have } G_w'(y_\omega)= \lambda^{|w|}.
\end{equation} 
}
We call \eqref{Eq Livsic condition} the \emph{Livsic Condition}, after the well known work of Livsic \cite{Livsic1972cohom} in dynamical systems. Roughly speaking, the Theorem of Livsic allows one to solve global cohomological equations from periodic data, for {H}\"older regular observables. This was brought into the context of smooth conjugation of Anosov systems by de la Llave - Moriy\'on \cite{Lave1988mori}. However, Bousch-Jenkinson \cite{Bousch2002Jen} have shown that this fails without {H}\"older regularity of the potential.

The following Theorem provides IFS-like analogues of both the Livsic Theorem, and  of its failure in lower regularity.

\begin{thm} \label{Threom C}
Let $\widetilde{\Phi}= \lbrace g_0,g_1 \rbrace\subseteq C^s([0,1])$ be a hyperbolic IFS where $s\geq 1$, that satisfies the Livsic condition \eqref{Eq Livsic condition}. Then:
\begin{enumerate}
\item For $s=1$ there exists such $\widetilde{\Phi}$ that is not $C^1$-conjugate to a pseudo-affine IFS. 
\item If $s>1$ then $\widetilde{\Phi}$ is $C^s$-conjugate to a pseudo-affine IFS of class $C^s$ with slope $\lambda$.  
\end{enumerate}
\end{thm}

We note that while Theorem \ref{Theorem A} does answer Question \ref{Question_Hochman}, there are some interesting problems that remain. For instance, we did not treat the case when the IFS has more than two generators, that has additional technicalities to deal with. Also, it would be interesting to understand more in detail how pseudo-affinity is manifested in terms of the scaling function. (Observe that the examples in Theorem \ref{Theorem A} - part b. are not $C^s$-conjugated to self-similar and thus their scaling functions are non-constant.)

We end this introduction with a few remarks about the implications Theorem \ref{Theorem A} has on the Fourier decay problem for fractal measures. Recall that a probability measure $\nu$ on $\mathbb{R}$ is called a Rajchman measure if its Fourier transform satisfies $\lim_{|q|\rightarrow \infty} \widehat{\nu}(q)=0$.  Combining recent breakthroughs on the Fourier decay problem \cite{bremont2019rajchman, varju2020fourier, li2019trigonometric, algom2020decay, algom2023polynomial, algom2024van, baker2024polynomial} we know that if a $C^\omega (\mathbb{R})$ self-conformal IFS admits a non-Rajchman self-conformal measure  then: The IFS must be self-similar; It must have contraction ratios that are all powers of some $r\in (0,1)$ such that $r^{-1}$ is a Pisot number; And, up to affine conjugation, all translates are in $\mathbb{Q}(r)$. Now, it is natural to ask about  analogues of this result in $C^r (\mathbb{R})$-regularity, $1\leq r\leq \infty$ - see e.g. \cite[Problem 3.14]{sahlsten2023fourier}. The contribution of Theorem \ref{Theorem A} to this line of research is that is shows the non-triviality of the problem: There are self-conformal measures that are neither self-similar, smoothly conjugate to self-similar, or non-linear. In particular, none of the methods in the cited papers can currently deal with the Rajchman problem for such measures.

\subsection{Organization of the paper}
In Section \ref{sec_preliminaries} we summarize preliminary content about regularity, Cantor sets and iterated function systems, that will be used throughout the text. In Section \ref{sec_pseudo-affine_IFS} we study the class of pseudo-affine IFS, providing a classification in Theorem \ref{thm_B_conjugation}. The main technical step is Proposition \ref{prop_classification_3}, that shows how to construct a Cantor set starting from its proportion functions. In Section \ref{sec_proof_Thm_A} we prove Theorem \ref{Theorem A}. The proof of Theorem \ref{Threom C} is given through Sections \ref{sec_Livshitz-I} and \ref{sec_Livshitz-II}.

\subsection*{Acknowledgements and Comments}
We wish to thank Jairo Bochi for useful discussions. We also thank Zhiren Wang for many helpful remarks. All authors are supported by Grant No. 2022034 from the United States - Israel Binational Science Foundation (BSF), Jerusalem, Israel; A.A. is also supported by  the Israel Science Foundation (Grant No. 392/25),   NSF-BSF Grant No. 2024692; F. RH. is also supported by NSF Grant no. 2453688.
Finally, we want to note that A. Erchenko and Y. Mazor have simultaneously and independently obtained results that are in connection with the ones presented in the current paper. 
 

\section{Preliminaries} \label{sec_preliminaries}

\subsection{Regularity} \label{Section reg}
Given $0<\alpha\leq 1$, we say that a map $g:[0,1]\to\R$ is \emph{$\alpha$-{H}\"older} if there exists a constant $A>0$ such that 
$$|g(u)-g(t)|\leq A\cdot|u-t|^{\alpha},\, \forall\ 0\leq t,u\leq 1.$$
We call $\alpha$ and $A$ the \emph{{H}\"older exponent} and \emph{{H}\"older constant}, respectively. For $\alpha=1$ this condition is called \emph{Lipschitz}. Given an integer $r\geq 1$, we say that a map $g:[0,1]\to\R$ is of \emph{class $C^{r,\alpha}$} if its derivatives up to order $r$ exist and are continuous, and the $r$-th derivative $\partial^r g$ is an $\alpha$-{H}\"older function. We call the parameter $s=r+\alpha\in[1,+\infty)\cup\{\infty\}$ the \emph{regularity} of $g$ (where $s=\infty$ stands for \emph{smooth}). We will simply say that $g$ is of class $C^s$. However, observe that a map of class $C^{r,1}$ has regularity $s=r+1$ although it does not mean it is of class $C^{r+1,0}$.  

If $X\subset[0,1]$ is a non-empty closed subset, \emph{a map $g:X\to\R$ is of class $C^s$ in the sense of Milnor} if for every $x\in X$ there exists an open neighborhood $U$ of $x$ and a map $G:U\to\R$ of class $C^s$ such that $G(z)=g(z)$, $\forall\ z\in U\cap X$. By compactness, it is possible to see that there always exists a $C^s$-extension $G:U\to\R$ of $g:X\to\R$ on an open neighborhood $U$ including $X$, and that $U$ can be taken to be $U=[0,1]$.   

\subsection{Binary Cantor Sets and Dynamical Proportions} \label{Section: binary cantor set}
Consider the product space $\{0,1\}^{\N}$ endowed  with the product topology, and with the lexicographic order induced by declaring $0<1$. A \emph{binary Cantor set} $X\subseteq [0,1]$  is the image of an \emph{order preserving} (with respect to the natural order on $[0,1]$) continuous embedding $\{0,1\}^{\N}\to[0,1]$, such that
 $$(0)_{n\geq 1}\in \{0,1\}^{\N} \mapsto 0\in [0,1], \text{ and } (1)_{n\geq 1}\in \{0,1\}^{\N} \mapsto 1\in [0,1].$$ 
Given  $a=(a_n)_{n\geq 1} \in \{0,1\}^{\N}$  we denote by $x_a$ its image in $[0,1]$. We emphasize that, in general, $X$  need not be a self-conformal set.

Let $\mathcal{W} := \lbrace 0,1\rbrace^*$ be the words in the alphabet $\{0,1\}$, and denote the \emph{empty word} by $e$. For every $w=w_1\cdots w_n\in \mathcal{W}$ we denote by $X_w$ the \emph{cylinder} determined by $w$, which is defined by
$$X_w := \lbrace x_a\in X:\, a_i =w_i\, \forall i=1,\dots,n\rbrace.$$
Define $K_w:=\textsl{Conv}(X_w)$, the \emph{convex hull} of the cylinder $X_w$. Then $K_w\subset[0,1]$ is a closed interval. We can thus write
$$X=\bigcap_{n\geq 0}\left(\bigcup_{|w|=n}K_w \right).$$
The connected components of $[0,1]\setminus X$ are open intervals called \emph{gaps}. We label the gaps via \newline $\{I_w:w\in\mathcal{W}\}$, which is given by setting 
$$I_w=K_w\setminus \left( K_{w0}\cup K_{w1} \right), \text{ for every } w\in\mathcal{W}.$$  

We will define a total order $\prec$ on the set $\mathcal{W}\cup\{0,1\}^\N$ by declaring that $w0(1)^\infty\prec w\prec w1(0)^\infty$, for every finite word $w=w_1\cdots w_n$. If $X$ is any binary Cantor set in $[0,1]$ with cylinders and gaps labeled as above, then  for every $u,w\in\mathcal{W}$ and  $a,b\in\{0,1\}^\N$, 
\begin{align*}
\label{lemma_order}
& u \prec w\ \text{if and only if the gap}\ I_u\ \text{is on the left of the gap}\ I_w, \\
& a \prec b\ \text{if and only if}\ x_a<x_b, \\
& w \prec a\ \text{if and only if the gap}\ I_w\ \text{is on  the left of the point}\ x_a, \\
& a \prec w\ \text{if and only if the point}\ x_a\ \text{is on the left of the gap}\ I_w.
\end{align*}

We are interested in the following quantities associated to a binary Cantor set, that encode information about its geometric structure. 
\begin{defn}
\label{defn_dynamical_proportions}
Given a binary Cantor set $X\subset[0,1]$ its associated pair of \emph{dynamical proportions} are the functions $\lambda_i:\W\to(0,+\infty)$, $i=0,1$ defined by 
$$\lambda_i(w)\coloneqq \frac{|I_{iw}|}{|I_w|},\ \forall\ w\in\W,\ \text{where}\ |I_w|\ \text{denotes the length of the gap}\ I_w\subset[0,1].$$
\end{defn}
We remark that, while different, this definition is motivated by Sullivan's scaling function \cite{Sullivan1987ratio, BedfordFisher1997Ratio}. The pair of dynamical proportions associated to a binary Cantor set will be our main tool to construct IFSs and determine their regularity in Section \ref{sec_pseudo-affine_IFS}. We proceed to prove Lemma \ref{lemma_Cantor_proportions} below, that will be used in the next section.

\begin{defn}
Given a pair of functions $\lambda_i:\W\to(0,+\infty)$, $i=0,1$, its associated \emph{cocycle} is the function $\Psi:\mathcal{W}\to(0,+\infty)$ given by the expression 
\begin{equation}
\label{equation_defn_Psi}
\Psi(w_1\cdots w_n)=\left(\prod_{i=1}^{n-1}\lambda_{w_{i}}(w_{i+1}\cdots w_n)\right)\cdot\lambda_{w_n}(e),\ \forall\ w\in\mathcal{W}.
\end{equation}  
\end{defn}
It is direct to check that, in the case where $\{\lambda_0,\lambda_1\}$ are the dynamical proportions of a binary Cantor set, then 
$|I_w|=\Psi(w)\cdot |I|$ $\forall\ w\in\W$, where $|I|$ is the length of the gap labeled with the empty word $e$. Observe that if $X$ has \emph{zero Lebesgue measure} then the sum of the length over all the gaps equals one and hence
\begin{equation}
\label{equation_sum_Psi}
\sum_{w\in\mathcal{W}}\Psi(w)=1/|I|>1. 
\end{equation}

\begin{lem}
\label{lemma_Cantor_proportions}
Let $\lambda_i:\mathcal{W}\to(0,+\infty)$, $i=0,1$ be any pair of functions and let $\Psi:\mathcal{W}\to(0,+\infty)$ be defined by the formula in equation \eqref{equation_defn_Psi} above. If  
$$1<\sum_{w\in\mathcal{W}}\Psi(w)<+\infty$$
then there is a unique zero Lebesgue measure binary Cantor set $X$ such that  $\{\lambda_0,\lambda_1\}$  are its pair of dynamical proportions. Furthermore, the length of the gaps in $X$ are given by the formula 
$$|I_w|=\Psi(w)\cdot\left(\sum_{u\in\mathcal{W}}\Psi(u)\right)^{-1},\ \forall\ w\in\mathcal{W}.$$
\end{lem}

\begin{proof}
Given the functions $\{\lambda_0,\lambda_1\}$ and its associated the cocycle $\Psi$ define $$L=\left(\sum_{w\in\mathcal{W}}\Psi(w)\right)^{-1},$$ which by hypothesis is a positive quantity smaller than $1$. Using the total order $\prec$ in the set $\mathcal{W}\cup\{0,1\}^\N$ that we explained above, define $\Theta:\{0,1\}^\N\to[0,1]$ by the formula
\begin{equation*}
\Theta(a)=\sum_{w\prec a}\Psi(w)\cdot L.
\end{equation*}
Let's check that  $\Theta:(0)_{n\geq 1}\mapsto 0$ and $\Theta:(1)_{n\geq 1}\mapsto 1$, that it is order preserving, injective and continuous. Since $(0)_{n\geq 1}$ is the minimum and $(1)_{n\geq 1}$ is the maximum for the order defined above, we have that 
\begin{align*}
& \Theta \left( (0)_{n\geq 1} \right)=\sum_{w\prec(0)_{n\geq 1}}\Psi(w)\cdot L=\left(\sum_{\emptyset}\Psi(w)\right)\cdot L=0,  \\
& \Theta \left( (1)_{n\geq 1} \right)=\sum_{w\prec(1)_{n\geq 1}}\Psi(w)\cdot L=\left(\sum_{w\in\mathcal{W}}\Psi(w)\right)\cdot L=1.
\end{align*}
If $a$ and $b$ are sequences satisfying $a<b$ for the lexicographic order, then $\{w:a\prec w\prec b\}\neq\emptyset$ and thus 
$$\Theta(a)=\sum_{w\prec a}\Psi(w)\cdot L<\sum_{w\prec a}\Psi(w)\cdot L+\sum_{a\prec w \prec b}\Psi(w)\cdot L=\sum_{w\prec b}\Psi(w)\cdot L=\Theta(b),$$ 
from where we can see that $\Theta$ is injective and preserves the respective orders in the domain and range. To see that $\Theta$ is continuous,
observe that if $a=(a_n)_{n\geq 1}$ and $b=(b_n)_{n\geq 1}$ satisfy $a_n=b_n$ for $n=1,\dots,N$ for some $N>0$, then all the words satisfying $a\prec w\prec b$ have length $|w|>N$. Thus
$$\left\vert\Theta(b)-\Theta(a)\right\vert=\left\vert\sum_{a\prec w\prec b}\Psi(w)\cdot L\right\vert\leq\sum_{a\prec w\prec b}|\Psi(w)|\cdot L\leq \sum_{|w|>N}\Psi(w)\cdot L.$$
Since $\sum_{w\in\mathcal{W}}\Psi(w)<+\infty$ then the right hand side in the expression above converges to zero when $N\to+\infty$, from where we see that $\Theta$ is continuous.

From the previous paragraph we obtain that $\Theta:\{0,1\}^\N\to[0,1]$ is a continuous embedding and its image $X=\Theta(\{0,1\}^\N)$ is a binary Cantor set in the unit interval. For every finite word $w=w_1\cdots w_n$ define the sequences $a_w=w01^{\infty}$ and $b_w=w10^\infty$. The sequences $a_w$ and $b_w$ are consecutive elements for the lexicographic order in $\{0,1\}^\N$ and thus they are the boundary points of the gap $I_w$. The word $w$ is the only element of $\mathcal{W}$ satisfying $a_w\prec w\prec b_w$ for the total order in $\mathcal{W}\cup\{0,1\}^\N$. It follows that
$$|I_w|=|\Theta(b_w)-\Theta(a_w)|=\sum_{a_w \prec u\prec b_w}\Psi(u)\cdot L=\Psi(w)\cdot L$$
and hence, using expression \eqref{equation_defn_Psi}, we see that the proportions of the Cantor set are given by 
$$\lambda_i(w)=\frac{|I_{iw}|}{|I_w|}=\frac{\Psi(iw)}{\Psi(w)}=\lambda_i(w),\ \forall\ w\in\mathcal{W}.$$
Finally, to check that $X$ has zero Lebesgue measure, observe that
$$\textsl{Leb}(X)=1-\sum_{w\in\mathcal{W}}|I_w|=1-\left(\sum_{w\in\mathcal{W}}\Psi(w)\right)\cdot L=0.$$
\end{proof}

\subsection{Hyperbolic structures on Cantor sets}

Consider $X=\{0,1\}^{\N}$ endowed with the lexicographic order. A \emph{$C^s$-atlas on $X$}, $s\geq 0$, is a family $\mathcal{D}=\{\varphi_\alpha:X_{w_\alpha}\to\R\}_\alpha$ of continuous, order preserving, embeddings of the cylinders of $X$ into $\R$, where $(i)$- $\{X_{w_\alpha}\}_\alpha$ forms an open cover of $X$ and $(ii)$- the change of coordinates are order preserving $C^s$-diffeomorphisms. A \emph{$C^s$-differentiable structure} on $X$ is a maximal $C^s$-atlas, and two structures $\mathcal{D}$, $\mathcal{D}'$ are \emph{$C^k$-equivalent}, $k\geq 0$, if there exists a $C^k$-diffeomorphism $(X,\mathcal{D})\to(X,\mathcal{D}')$. We denote by $[\mathcal{D}]$ the set of all structures $C^s$-equivalent to $\mathcal{D}$.  

A $C^s$-differentiable structure $\mathcal{D}$ on $X$, $s\geq 1$, is called \emph{hyperbolic} if the shift map $\sigma:(X,\mathcal{D})\to(X,\mathcal{D})$ is a \emph{$C^s$-expanding map}, i.e. it has derivative greater than one at every point. In \cite{Sullivan1987ratio} and \cite{BedfordFisher1997Ratio} there is a classification of the equivalence classes of hyperbolic structures in $r=1+\alpha$ regularity. The classifying invariant is a {H}\"older function $\rho:Y\to\Delta$ called \emph{scaling function}, defined over an abstract Cantor set $Y$ called the \emph{dual Cantor set} of $X$ and with image in the unit simplex $\Delta=\{(t_1,t_2,t_3)\in\R^3:0\leq t_i\leq 1,\ t_1+t_2+t_3=1\}$. 

The construction of the scaling function is as follows: Let $\mathcal{D}$ be a hyperbolic $C^{1+\alpha}$-structure on $X$ and take a globally defined chart $\varphi:X\to\R$ embedding the Cantor set in the real line with extreme points $t=0$ and $t=1$. Since $\varphi(X)\subset\R$ is a binary Cantor set with a labeling on its cylinders, this embedding is completely determined by knowing the ratios $r(w)=(r_0(w),r_{\mathrm{gap}}(w),r_1(w))\in\Delta$ for every finite word $w$, where:
\begin{align*}
r_i(w)=\frac{\diam(\varphi(X_{wi}))}{\diam(\varphi(X_{w}))}\text{ for }i=0,1\ \text{ and }\ 
r_{\mathrm{gap}}(w)=1-\frac{\diam(\varphi(X_{w0}))+\diam(\varphi(X_{w0}))}{\diam(\varphi(X_{w}))}.
\end{align*}
The key point is that, if the shift map is $C^{1+\alpha}$-expanding, then it is possible to use the bounded variation of the derivative to prove that for every sequence $(a_n)_{n=1}^{\infty}$ with $a_n\in\{0,1\}$ there exists the limit $\lim_{n\to\infty} r(a_n\cdots a_1)\in\Delta$. Observe the detail that the sequence of words indexing this limit is constructed by the procedure of adding the letters to the left. (The limit could not exists if we add letters to the right.)

The set $\mathcal{W}$ of finite ordered strings in the alphabet $\{0,1\}$ can be thought of in two different ways: $(i)$- As being constructed from the empty word by adding letters to the right, $(ii)$- As being constructed from the empty word by adding letters to the left. Each construction allows to order the words in $\mathcal{W}$ in a one-sided tree, right-sided in case $(i)$ and left-sided in case $(ii)$, and each tree has a boundary Cantor set: $\prod_{i=1}^\infty\{0,1\}$ in case $(i)$ and $\prod_{-\infty}^{i=-1}\{0,1\}$ in case $(ii)$. Given now the Cantor set $X$ and the labeling $\{X_w:w\in\mathcal{W}\}$, observe that the boundary of the right-sided tree (case $(i)$) can be identified with $X$ itself. The boundary of the left-sided tree (case $(ii)$) will be denoted by $Y$ and will be called the dual Cantor set of $X$. It is the set over which the limits above are naturally defined.   

\begin{thm}[Classification of hyperbolic $C^{1+\alpha}$-structures, \cite{Sullivan1987ratio}-\cite{BedfordFisher1997Ratio}.]\label{thm_Sullivan-1}
Let $\mathcal{D}$ be a $C^{1+\alpha}$-structure on $X$ with \emph{bounded geometry}. Then $\sigma:(X,\mathcal{D})\to(X,\mathcal{D})$ is a $C^{1+\alpha}$-expanding map if and only if there exist the limit $\rho(y)=\lim_{n\to-\infty}r(y_{n}\cdots y_{-1})$ for all $y\in Y$, it is independent of the chart $\varphi$, and the limit function $\rho:Y\to\Delta$ is $\alpha$-{H}\"older. Moreover, the map that assigns to each $\mathcal{D}$ its scaling function $\rho:Y\to\Delta$ induces a one to one correspondence
$$\Diff^{\ 1+\alpha}(X,\sigma)\coloneqq\left\{[\mathcal{D}]:\ \sigma:(X,\mathcal{D})\to(X,\mathcal{D})\ \text{is}\ C^{1+\alpha}-\text{expanding}\ \right\}\xrightarrow[]{(1:1)} C^{\alpha}(Y,\Delta).$$
\end{thm} 

Moving to the case of regularity $s\geq 2$, a description of the $C^s$-equivalence classes of $C^s$-structures making the shift $C^s$-expanding becomes more subtle. In \cite{Sullivan1987ratio}, \cite{BedfordFisher1997Ratio} they prove the following fundamental result:

\begin{thm}[Rigidity, \cite{Sullivan1987ratio}-\cite{BedfordFisher1997Ratio}]
Any hyperbolic $C^s$-structure $\mathcal{D}$ satisfies $\diff^{\ 1}(X,\mathcal{D})=\diff^{\ s}(X,\mathcal{D})$, $\forall s\geq 1+\alpha$.
\end{thm}

The rigidity theorem implies that, for every $s\geq 1+\alpha$: $(i)$- if a $C^{1+\alpha}$-atlas $\mathcal{D}$ has a refinement of class $C^s$ for which the shift is $C^s$-expanding then $\mathcal{D}$ is already of class $C^s$; $(ii)$- the correspondence $\Diff^{\ s}(X,\sigma)\to C^{\alpha}(X,\Delta)$ sending $[\mathcal{D}]\mapsto\rho$ is injective.

\subsection{Some properties of self-conformal IFSs} \label{Section: self conformal}

Recall the definitions of self-conformal and hyperbolic IFSs from Section \ref{Section: Intro}, and let $\Phi = \lbrace f_0, f_1 \rbrace\subseteq C^1 ([0,1])$ be a hyperbolic IFS. For every $a\in \lbrace0,1\rbrace^\mathbb{N}$ define
$$x_a = \lim_{n\rightarrow \infty} f_{a_1}\circ \circ \circ f_{a_n} (0).$$
It is well known that $x_a\in X=X_\Phi$ and that this is map is actually a bijection. In fact, this map shows that $X$ is a binary Cantor set in the sense of Section \ref{Section: binary cantor set}. Furthermore, it is easy to see that, by strong separation, $\text{Leb}(X)=0$.

Let us discuss some more standard properties of hyperbolic Cantor sets; as before, all proofs can be worked out from e.g. \cite[Chapter 2]{bishop2013fractal}. For $w\in \mathcal{W}$ let $F_w$ be the cylinder map
$$F_w = f_{w_{1}}\circ \circ \circ f_{w_n}.$$
Note that if $y=F_w(x_a)$ then 
$$y =  x_{w_1\cdots w_n a_1 a_2 a_3 \cdots}.$$
Also, $F_w$ has a unique fixed point given by 
$$x_{w^\infty} =x_{(w_1\cdots w_n)(w_1\cdots w_n)(w_1\cdots w_n)\cdots}.$$  

Next, we have the usual labeling $\{I_w\}_{w\in\W}$ of the components of $[0,1]\setminus X$ defined by setting $I_w\coloneqq K_w\setminus(K_{w0}\cup K_{w1})$, $\forall\ w\in\W$. Then we always have 
$$F_w(I_u)=I_{wu}, \, \forall\ w,u\in\W.$$ 
Observe that the dynamical proportions of the invariant Cantor set must satisfy $0<\lambda_i(w)<1$, $\forall$ $w\in\mathcal{W}$ and $i=0,1$. We also note that  
$$|I_w|<\mu^{|w|},\, \forall w\in\mathcal{W},$$ 
where $0<\mu<1$ is any constant such that  $f_i'(t)<\mu<1$, $\forall\ 0<t<1$, $i=0,1$. Hence, $|I_w|$ converges exponentially fast to zero when $|w|\to\infty$.

Finally, we note that if $\widetilde{\Phi}=\lbrace g_0,g_1\rbrace$ is another hyperbolic IFS then there exists a homeomorphism $h:[0,1]\to[0,1]$ such that $h(K_\Phi)=K_{\widetilde{\Phi}}$ and $h\circ F_w(x)=\widetilde{F}_w\circ h(x)$, $\forall\ x\in X_\Phi$.

\begin{figure}[t]
\begin{center}
\includegraphics[width=\textwidth]{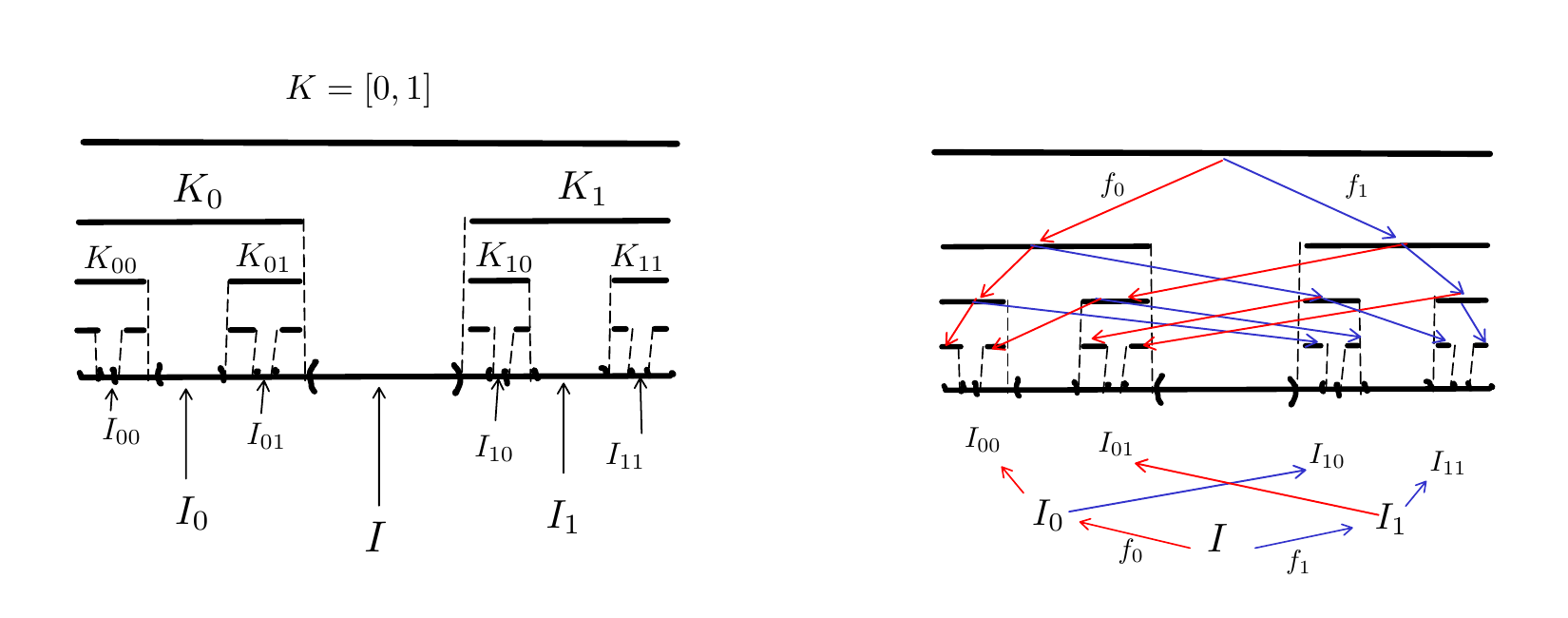}
\caption{On the left: Labeling of cylinders and gaps. On the right: The action of the IFS.}
\label{fig_IFS_action}
\end{center}
\end{figure}

\section{pseudo-Affine hyperbolic IFS.}
\label{sec_pseudo-affine_IFS}
Recall from \eqref{defn_pseudo-affine_IFS} that a self-conformal IFS $\Phi$ of class $C^s$, where $1\leq s\leq \infty$ or $s=\omega$, is \emph{pseudo-affine} if there exists $0<\lambda<1/2$ such that $f_i'(x)=\lambda$ for all $x\in X_\Phi$ and $f\in \Phi$. In this section we give a complete classification of the $C^s$-conjugacy classes of pseudo-affine IFSs.

\subsection{Asymptotically constant dynamical proportions}
We begin by characterizing pseudo-affinity in terms of dynamical proportions.

\begin{defn}
We say that a binary Cantor set $X$ with gap-labelling $\{I_w\}_{w\in\mathcal{W}}$  has \emph{asymptotically constant dynamical proportions} if there exists $0<\lambda<1/2$ such that 
\begin{equation*}
\lambda_i(w)=\lambda+\theta_i(w),\ \forall\ w\in\mathcal{W}\ \text{for both}\ i=0,1,
\end{equation*}
and 
$$\lim_{n\to\infty}\theta_i(w_1\cdots w_n)=0,\quad \forall\ (w_n)_{n\geq 1} \in \lbrace0,1\rbrace^\mathbb{N}.$$ 
\end{defn}
Note that attractors of self-similar IFSs with a uniform contraction ratio $\lambda$  have  \textit{constant} dynamical proportions
$\lambda_i(w)=\lambda$, $\forall\ w\in\mathcal{W}$. 


\begin{prop}\label{prop_classification_1}
Let $\Phi = \lbrace f_1, f_2 \rbrace$ be a  $C^1$ hyperbolic IFS. Then the following properties are equivalent:
\begin{enumerate}
\item $\Phi$ is pseudo-affine;
\item $X=X_\Phi$ has asymptotically constant dynamical proportions.
\end{enumerate}
\end{prop}

\begin{proof}
Since the contractions $f_i:[0,1]\to\R$, $i=0,1$ are of class $C^1$ the proportions are given by 
\begin{equation*}
\lambda_i(w)=\frac{|I_{iw}|}{|I_w|}=\frac{1}{|I_w|}\int_{I_w}f_i'(t)dt.
\end{equation*}
\begin{itemize}
\item Suppose $f_i'(x)=\lambda$ for every $x\in X$, $i=0,1$.  Write $f_i'(t)=\lambda+\eta_i(t)$, $\forall\ 0\leq t\leq 1, i=0,1$. Then
\begin{equation*}
\lambda_i(w)=\lambda+\theta_i(w),\text{ where }\theta_i(w)=\frac{1}{|I_w|}\int_{I_w}\eta_i(t)dt.
\end{equation*}
Let's see that $\theta_i(w)\to 0$ \emph{uniformly} in $|w|\to\infty$. Fix some $\varepsilon>0$. Since $f_i'$, $i=0,1$ are both continuous then there exists $\delta>0$ such that $|f_i'(x)-f_i'(y)|<\varepsilon$ whenever $|x-y|<\delta$. Since $|I_w|\to 0$ when $|w|\to\infty$ (cf. Section \ref{Section: binary cantor set}), there  exists $n_0>0$ such that $|I_w|<\delta$ when $|w|>n_0$. Then $|\theta_i(w)|<\varepsilon$ for every $w$ with length $|w|>n_0$.

\item Reciprocally, assume that $\lambda_i(w)=\lambda+\theta_i(w)$ with $\theta_i(w)\to 0$ when $|w|\to\infty$. Recall that for every finite word $w=w_1\cdots w_n\in\mathcal{W}$ the cylinder $F_w:[0,1]\to[0,1]$ has a unique fixed point $x_w\in[0,1]$, that belongs to $X$ and has a periodic point coded by 
$$(w_1\cdots w_n)(w_1\cdots w_n)\cdots.$$ 
Recall that $\{x_w:w\in\mathcal{W}\}$ is  dense  in $X$. It suffices to show that $f_i'$ is constant on this set. Define $w^k=\prod_{i=1}^k(w_1\cdots w_n)$. Since $f_i'$ is continuous, its derivative on $x_w$ can be expressed by
\begin{equation}
f_i'(x_w)=\lim_{k\to\infty}\frac{1}{|I_{w^k}|}\int_{I_{w^k}}f_i'(t)dt=\lim_{k\to\infty}(\lambda+\theta_i(w^k))=\lambda.
\end{equation}
The claim is proved.
\end{itemize}
\end{proof}


The next Proposition shows that the degree of regularity of the IFS is expressed in the convergence ratio of $|\theta_i(w)|\to 0$ when $|w|\to \infty$. 

\begin{prop}\label{prop_calssification_2}
Let $\Phi$ be a pseudo-affine hyperbolic IFS of class $C^1$ and let its dynamical proportions be 
$\{\lambda_i(w)=\lambda+\theta_i(w),\ w\in\mathcal{W}\}_{i=0,1}$. Then $\theta_i(w)\to 0$ uniformly in $|w|\to\infty$. Moreover:

\begin{enumerate}
\item If $\Phi$ is $C^{r,\alpha}$ for $1\leq r <\infty$, $0<\alpha\leq 1$, there exists $C>0$ such that $|\theta_i(w)|<C\cdot|I_w|^{r-1+\alpha}$, $\forall\ w\in\mathcal{W}$;
\item If $\Phi$ is $C^\infty$ then for every $1\leq r <\infty$ there exists $C(r)>0$ such that $|\theta_i(w)|<C(r)\cdot|I_w|^{r}$, $\forall\ w\in\mathcal{W}$.
\end{enumerate}
\end{prop}

\begin{proof}
For every  $w \in \mathcal{W}$ let us abuse notation and write $a_w<b_w$ for the extremities of the gap $I_w$. We retain the  notations as in the  proof of Proposition \ref{prop_classification_1}. In particular, write   $\theta_i(w)$ as the integral over the gap $I_w$ of $\eta_i(s):=f_i'(s)-f_i'(a_w)$, where $\lambda=f_i'(a_w)$. 
Now, that $|\theta_i(w)|\to 0$ uniformly in $|w|\to\infty$ has already been shown along the proof of Proposition \ref{prop_classification_1}. In addition we have:
\begin{enumerate}
\item If $\Phi$ is of class $C^{r,\alpha}$ with $1\leq r <\infty$ and $0<\alpha\leq 1$ then $\eta_i$ is of class $C^{r-1,\alpha}$ and $\eta_i^{(k)}(x)=0$, for every $k=0,\dots,r-1$ and $x\in X$. This implies that 
$$|\eta_i(s)|\leq C\cdot|s-a_w|^{r-1+\alpha} \text{ for every } a_w<s<b_w, \text{  where } C=A\cdot(\prod_{i=1}^{r-1}(\alpha+i))^{-1},$$
and  $A = \text{ the H\"older constant of } \eta_i^{(r-1)}.$ Integrating on $I_w=[a_w,b_w]$ we obtain that 
$$|\theta_i(w)|\leq C\cdot|I_w|^{r-1+\alpha}.$$
\item If $\Phi$ is of class $C^\infty$ then it is of class $C^{r,1}$ for every $1\leq r<\infty$, so the statement follows by applying the previous case on each degree of regularity. 
\end{enumerate}
\end{proof}

\begin{rem}
\label{rmk_exponential_convergence}
In any case, in regularity $s=r+\alpha>1$ we have that $|\theta_i(w)|\to 0$ exponentially fast in $|w|\to\infty$. To check this, recall that $|I_w|\leq K\cdot \mu^{|w|}$ for some $K>0$ and $0<\mu<1$ (cf. Section \ref{sec_preliminaries}), and hence in regularity $s=r+\alpha>1$ we have $|\theta_i(w)|<(CK^{s-1})\cdot \mu^{(s-1)|w|}$ that converges to zero exponentially fast with $|w|\to\infty$. 
\end{rem}


\subsection{The key construction}
The next proposition, which is essentially the converse of the implication discussed in Remark \ref{rmk_exponential_convergence}, is the key to our construction of pseudo-affine IFS in Theorem \ref{Theorem A}.

\begin{prop}\label{prop_classification_3}
Fix $0<\lambda< 1/2$. Suppose  a given pair of functions $\theta_i:\mathcal{W}\to\R$, $i=0,1$ satisfy that:
\begin{enumerate}[(a)] 
\item $|\theta_i(w)|\to 0$ uniformly in $|w|\to\infty$; and,

\item $0<\lambda+\theta_i(w)<1$ $\forall$ $w\in\mathcal{W}$
\end{enumerate}
Then there exists a pseudo-affine hyperbolic IFS $\Phi$ of class $C^1$ with dynamical proportions $\lambda_i=\lambda+\theta_i$. Moreover: 
\begin{enumerate}
\item If $|\theta_i(w)|<C\cdot|I_w|^{r-1+\alpha}$ for some $1\leq r<\infty$, $0<\alpha\leq 1$ and $C>0$, then $\Phi$ is of class $C^{r,\alpha}$;
\item If for every $1\leq r<\infty$ there exists $C(r)>0$ such that $|\theta_i(w)|<C(r)\cdot|I_w|^{r}$, then $\Phi$ is of class $C^\infty$;
\end{enumerate}
\end{prop}

\begin{proof}
Consider the pair $\lambda_i:\mathcal{W}\to(0,1)$, $i=0,1$ defined as $\lambda_i(w)=\lambda+\theta_i(w)$. We first verify that $\{\lambda_0,\lambda_1\}$ satisfies the admissibility condition of Lemma \ref{lemma_Cantor_proportions}:

\begin{cl}\label{claim_1-prop_classification_2}
We have
\begin{align*}
1<\sum_{w\in\mathcal{W}}\Psi(w)<\infty,\ \text{where}\ \Psi(w)=\prod_{k=1}^{|w|}\lambda_{w_{k}}(w_{k+1}\cdots w_{|w|}).
\end{align*}
\end{cl}

\begin{proof}[Proof of Claim \ref{claim_1-prop_classification_2}]
Pick $\varepsilon>0$ such that $\lambda+\varepsilon<1/2$. Since $\theta_i\to 0$ uniformly there exists $N_0>0$ such that if $|w|>N_0$ then $|\theta_i(w)|<\varepsilon$. Then there are constants $C_0,C_1>0$ such that
\begin{align*}
1<\sum_{n=0}^{\infty}
\sum_{|w|=n}\Psi(w) &=\sum_{n=0}^{\infty}\sum_{|w|=n}\left(\prod_{k=1}^{|w|}\lambda_{w_{k}}(w_{k+1}\cdots w_n)\right)\\
	&=C_0+\sum_{n=N_0}^\infty\sum_{|w|=n}\left(\prod_{k=1}^{N_0}\lambda_{w_{k}}(w_{k+1}\cdots w_n)\cdot\prod_{k=N_0+1}^{n}(\lambda+\theta_i(w_{k+1}\cdots w_n))\right)\\
	&\leq C_0+\sum_{n=N_0}^\infty\sum_{|w|=n}\left(\prod_{k=1}^{N_0}\lambda_{w_{k}}(w_{k+1}\cdots w_n)\cdot\prod_{k=N_0+1}^{n}(\lambda+\varepsilon)\right)\\
	&\leq C_0+C_1\cdot \sum_{n=N_0}^\infty\sum_{|w|=n}\left(\prod_{k=1}^{n}(\lambda+\varepsilon)\right)=C_0+C_1\cdot\sum_{n=N_0}^\infty[2(\lambda+\varepsilon)]^n<\infty. 
\end{align*}
\end{proof}

By Lemma \ref{lemma_Cantor_proportions}, there exists a binary Cantor set $X\subset[0,1]$ with a marking $\{I_w\}_{w\in\mathcal{W}}$ and proportions satisfying $|I_{iw}|/|I_w|=\lambda_i(w)$. We want to show that $X$ is the attractor of an hyperbolic IFS $\Phi=\left\lbrace f_0,f_1\right\rbrace$ of class $C^1$. To this end, we will first construct the derivatives $f_i':[0,1]\to\R$ and then obtain the branches of $\Phi$ by integration. Each derivative $f_i'\in C^0([0,1])$ will be obtained as an infinite sum
\begin{equation}\label{derivatives_of_branches-prop_classification_2}
f_i'(t)=\lambda+\sum_{w\in\mathcal{W}}\varphi_w^i(t),
\end{equation}
where each $\varphi_w^i:[0,1]\to\R$ is a continuous functions with support in $\overline{I_w}$ as in Figure \eqref{fig_prop_classification_2}.  The functions $\varphi_w^i$ will be obtained by re-scaling the domain and range of the graph of a fixed smooth function $\rho:[0,1]\to\R$ in the following manner: Fix a function $\rho:[0,1]\to[0,+\infty)$ of class $C^\infty$ satisfying that:
\begin{itemize}
\item $\frac{d^k}{d t^k}\rho(0)=\frac{d^k}{d t^k}\rho(1)=0$, for every $k\geq 0$;
\item $\int_0^1\rho(t)dt=1$.
\end{itemize}
Given an interval $J=[a,b]$ and a constant $\alpha\in\R$, define a smooth function $\rho_{(\alpha,J)}:\R\to\R$ with support in $J$ and $\int_J\rho(t)dt=\alpha\cdot|J|$ by 
\begin{equation*}
\rho_{(\alpha,J)}(t)=\alpha\cdot\rho\left(\frac{t-a}{b-a}\right).
\end{equation*}

\begin{figure}[t]
\begin{center}
\includegraphics[scale=0.5]{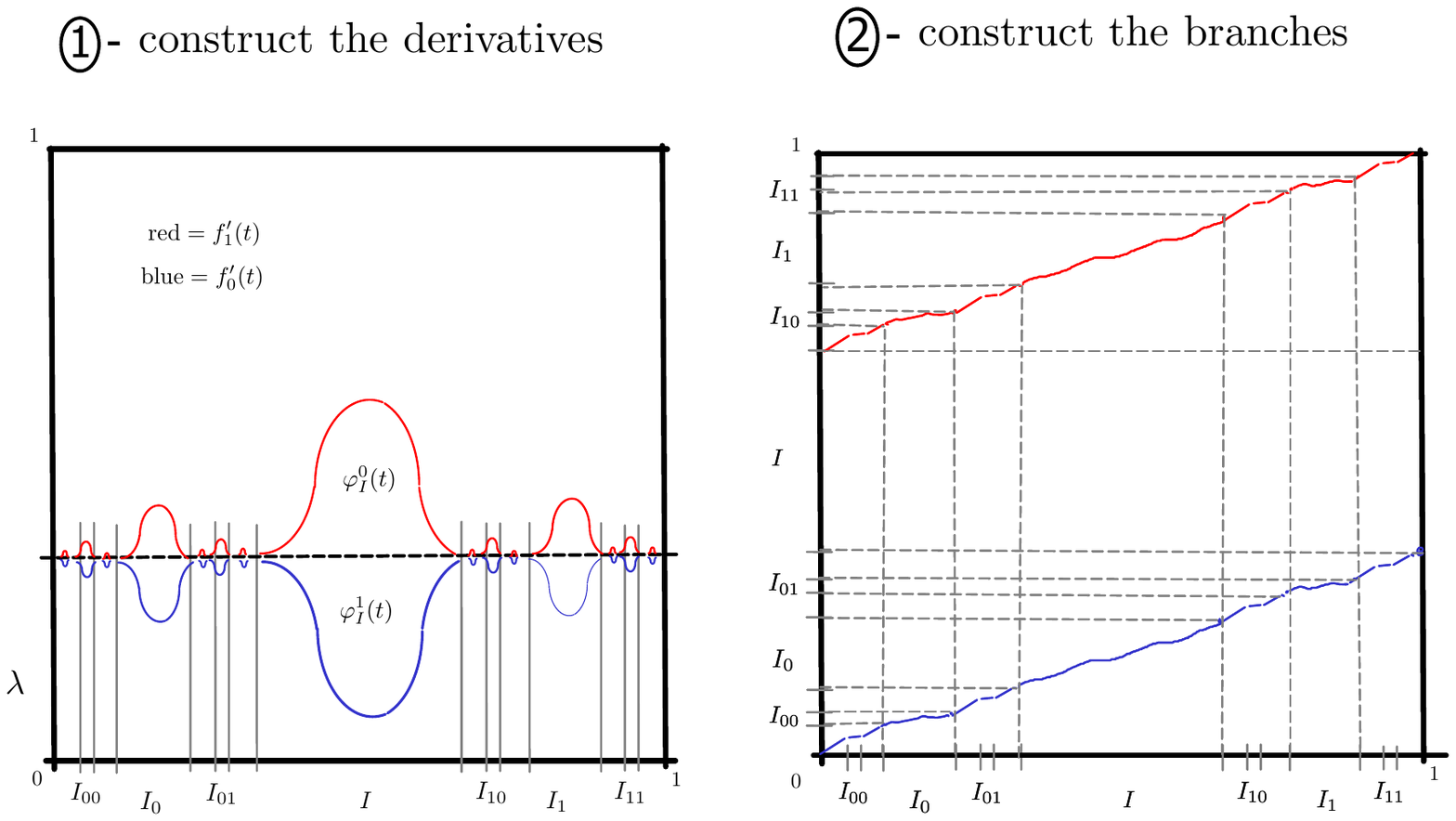}
\label{fig_prop_classification_2}
\caption{The proportions $\lambda_i:\mathcal{W}\to(0,1)$, $i=0,1$ define a unique Cantor set $X$ with a marking $\{I_w\}_{w\in\mathcal{W}}$ of its gaps such that $|I_{iw}|/|I_w|=\lambda_i(w)$. We see here how to define an IFS $\Phi=\left\lbrace f_0,f_1 \right\rbrace$ that preserves this cantor set.}
\end{center}
\end{figure}

\begin{cl}\label{claim_2-prop_classification_2}
\emph{
If we choose $\varphi_w^i\coloneqq\rho_{(\theta_i(w),I_w)}$ for each gap $I_w$ and $i=0,1$, then the pair of functions $\{f_0',f_1'\}$ given by the expression 
\begin{equation*}
f_i'(t)= \left\{ \begin{array}{ccc} \lambda+\varphi_w^i(t) & \text{if}\ & t\in I_w,\ w\in\mathcal{W};\\ 
\\ \lambda & \text{if}\ & t\in X. \\
\end{array} \right.
\end{equation*}
satisfies that:
\begin{enumerate}
\item Each $f_i':[0,1]\to\R$ is well-defined and continuous;
\item The maps $f_i:[0,1]\to[0,1]$ given by: \newline
\noindent $f_0(t)=\int_0^t f_0'(s)ds$; and, \\
\noindent $f_1(t)=\tau+\int_0^t f_1'(s)ds$, where $\tau=1-\int_0^1 f_1'(s)ds$;\\
form a  $C^1$ hyperbolic IFS $\Phi$ such that  $X = X_\Phi$.
\end{enumerate}
}
\end{cl}

\begin{proof}[Proof of Claim \ref{claim_2-prop_classification_2}]
The functions $f_i'$ are well-defined because they are set to be constant $=\lambda$ along the boundary of the domain of definition, which is the Cantor set $X$. If we write
\begin{equation*}
f_i'(t)=\lambda+\lim_{n\to\infty}\sum_{|w|\leq n}\varphi_w^i(t)
\end{equation*}
then for $n<m$ we have that:
\begin{equation*}
\left\Vert\sum_{|w|=n}^m \varphi_w^i \right\Vert_{\infty}
=\sup_{n\leq |w|\leq m}\left\{ \Vert\varphi_w^i\Vert_\infty\right\}
\leq \Vert\rho\Vert_\infty\cdot\sup_{n\leq |w|\leq m}\left\{|\theta_i(w)|\right\},
\end{equation*}
which converges to zero with $n\to\infty$ due to the uniform convergence $|\theta_i(w)|\to 0$. Hence the functions $f_i'$, $i=0,1$ are continuous, being the limit of a Cauchy sequence of continuous functions. 

For the $f_i:[0,1]\to[0,1]$, $i=0,1$ defined above it is clear that both are $C^1$ contractions satisfying $f_0(0)=0$ and $f_1(1)=1$. To check that the Cantor set $X$ is invariant under the IFS $\Phi=\left\lbrace f_0,f_1\right\rbrace$, consider a point $x_a\in X$ with coding $a=(a_n)_{n\geq 1}$ and consider the points $x_{ia}\in X,\ i=0,1$ where $ia=( ia_1 a_2 a_3 \cdots )$. Given a finite word $w=w_1\cdots w_n$ and a sequence $a= (a_n)_{n\geq 1}\in\{0,1\}^\mathbb{N}$ we denote $w<a$ if the gap $I_w$ is located to the left of the point $x_a$, or $a<w$ otherwise. Since the Cantor set is constructed by iteratively removing the gap $I_w$ from a closed interval $K_w$, $w\in\mathcal{W}$, then $iw<ia$ if and only if $w<a$, for both $i=0,1$. Using that the Cantor set $X$ has zero Lebesgue measure we get that the points $x_a, x_{0a}, x_{1a}$ are determined by the expressions
\begin{align*}
x_a         &=|x_a|   =\sum_{w< a}|I_w|=1-\sum_{a< w}|I_w|,\\
x_{0a}      &=|x_{0a}|=\sum_{w< 0a}|I_w|=\sum_{w< a}|I_{0w}|,\\
x_{1a}      &=|x_{1a}|=1-\sum_{1a< w}|I_w|=1-\sum_{a< w}|I_{1w}|.
\end{align*}

On the other hand, for each $i=0,1$ let $y_i=f_i(x_a)$. Then
\begin{align*}
y_0      &=\int_0^{x_a} f_0'(t)dt =\int_{[0,{x_a}]}\left(\sum_{w\in\mathcal{W}}(f_0')|_{I_w}\right)dt=\sum_{w< a}\int_{I_w}f_0'(t)dt,\\
y_1      &=\tau+\int_0^{x_a} f_1'(t)dt =1-\int_{[x_a,1]}\left(\sum_{w\in\mathcal{W}}(f_1')|_{I_w}\right)dt=1-\sum_{a< w}\int_{I_w}f_1'(t)dt;
\end{align*}
Since $\int_{I_w}\varphi_w^i(t)dt=\theta_i(w)|I_w|$ then by construction we have that 
$$\int_{I_w}f_i'(t)dt=\int_{I_w}(\lambda+\varphi_w^i(t))dt=(\lambda+\theta_i(w))|I_w|=\lambda_i(w)|I_w|.$$
It follows that
\begin{align*}
y_0      &=\sum_{w< a}\lambda_0(w)|I_w|=\sum_{w< a}|I_{0w}|=x_{0a},\\
y_1      &=1-\sum_{a< w}\lambda_1(w)|I_w|=1-\sum_{a< w}|I_{1w}|=x_{1a}.
\end{align*}
By the last two expressions we see that $y_i\in X$ for both $i=0,1$. Therefore $\Phi$ preserves $X$, so $X_\Phi \subseteq X$; that $X=X_\Phi$ follows by iterating this  argument to show that every point coded by the IFS belongs to $X$ (recall that $X$ is closed by definition). The proof is complete. 
\end{proof}

\begin{rem} In this Remark we study the regularity of the IFS $\Phi$ obtained above. Write $f_i'=\lambda+\lim_{n\to\infty}\sum_{|w|\leq n}\varphi_w^i$ where $i=0,1$. 

\begin{enumerate}
\item Suppose there exist $C>0$ and $1\leq r<\infty$, $0<\alpha\leq 1$ such that $|\theta_i(w)|<C\cdot|I_w|^{r-1+\alpha}$, $\forall w\in\mathcal{W}$. We first show  that for every $k=0,\dots,r-1$ the $k$-th derivative of $\sum_{|w|=1}^n \varphi_w^i$   is a Cauchy sequence. By construction we have that
\begin{equation}
\left\Vert\partial^k\varphi_w^i\right\Vert_\infty=\left\Vert\frac{d^k}{dt^k}\left[\theta_i(w)\rho\left(\frac{t-a_w}{b_w-a_w}\right)\right]\right\Vert_\infty\leq \frac{\vert\theta_i(w)\vert}{|I_w|^k}\cdot\left\Vert\partial^k \rho\right\Vert_\infty\leq C\cdot|I_w|^{r-1+\alpha-k},\ \forall\ k=0,\dots,r-1;
\end{equation}
then for $n<m$ we have 
\begin{equation*}
\left\Vert\sum_{|w|=n}^m \partial^k\varphi_w^i \right\Vert_{\infty}
=\sup_{n\leq |w|\leq m}\left\{ \Vert \partial^k\varphi_w^i\Vert_\infty\right\}
\leq C\cdot \sup_{n\leq |w|\leq m}\left\{ |I_w|^{r-1+\alpha-k}\right\},
\end{equation*}
which converges to zero when $n\to\infty$ for $k=0,\dots,r-1$. This implies that $f_i$ is of class $C^r$ and 
$$\partial^{k+1} f_i =\lim_{n\to\infty}\sum_{|w|\leq n}\partial^{k} \varphi_w^i,\ \forall\ k=0,\dots,r-1.$$ 

We now show that for each $i=0,1$ the derivatives $\partial^{r} f_i:[0,1]\to\R$ are $\alpha$-H\"older functions. First, observe that for each finite word $w$ the function $\partial^{r-1}\varphi_w^i$ is $\alpha$-H\"older on the closure of the gap $I_w=(a_w,b_w)$. To check this observe that $\partial^{r-1}\rho$ is an $\alpha$-H\"older function for some constant $A$, and so for two points $a_w\leq t,u\leq b_w$ we have that
\begin{equation*}
\vert\partial^{r-1}\varphi_w^i(u)-\partial^{r-1}\varphi_w^i(t)\vert\leq C\cdot|I_w|^\alpha\cdot\left\vert\partial^{r-1}\rho\left(\frac{u-a_w}{b_w-a_w}\right)-\partial^{r-1}\rho\left(\frac{t-a_w}{b_w-a_w}\right)\right\vert<(CA)\cdot\vert u-t\vert^{\alpha}. 
\end{equation*}  
Next, observe that for every $n\geq 1$ the function $\psi_n:[0,1]\to\R$ defined by $\psi_n=\sum_{|w|\leq n}\partial^{r-1}\varphi_w^i$ is $\alpha$-H\"older. To check this, take two points $0\leq t< u\leq 1$ on the support of $\psi_n$ and assume that $t$ and $u$ belong to different gaps $I_v=(a_v,b_v)$ and $I_w=(a_w,b_w)$, $|v|,|w|\leq n$. Then
\begin{align*}
\vert \psi_n(u)-\psi_n(t) \vert
& \leq\vert \partial^{r-1}\varphi_w^i(u)-\partial^{r-1}\varphi_w^i(a_w)\vert+\vert\partial^{r-1}\varphi_w^i(a_w)-\partial^{r-1}\varphi_w^i(b_v)\vert\\
& +\vert\partial^{r-1}\varphi_w^i(b_v)-\partial^{r-1}\varphi_w^i(t)\vert\\
& \leq (CA)\cdot\left(|u-a_w|^\alpha+|a_w-b_v|^\alpha+|b_v-t|^\alpha\right)\leq (K_\alpha CA)\cdot|t-u|^\alpha,
\end{align*}
where $K_\alpha$ is taken to satisfy $a^\alpha+b^\alpha+c^\alpha\leq K_\alpha\cdot (a+b+c)^\alpha$, for every $0,\leq a,b,c\leq 1$. In the same way, it is possible to check that $\vert \psi_n(u)-\psi_n(t) \vert\leq (K_\alpha CA)|t-u|^{\alpha}$ for every $t,u\in[0,1]$. Finally, since the sequence of functions $\psi_n:[0,1]\to\R$ converges uniformly to $\partial^{r}f_i$ and each $\psi_n$ is $\alpha$-H\"older with constant $K_\alpha CA>0$ (independent of $n\geq 1$) we conclude that $\partial^r f_i$ is $\alpha$-H\"older with the same constant. 

\item If the condition $|\theta_i(w)|<C(r)\cdot|I_w|^r$, $\forall\ w\in\mathcal{W}$ with $C(r)>0$ is satisfied for every $r\geq 1$, then by the previous item we have that the $r$-th derivative of $f_i$, $i=0,1$ exists and is continuous for every $r\geq 1$. Thus, the contractions $f_i$ are $C^\infty$. 
\end{enumerate}

\end{rem}
Taking tally of the previous claims and remarks, Proposition \ref{prop_classification_3} is proved.
\end{proof}


\subsection{$C^s$-conjugacy classes of pseudo-affine hyperbolic IFSs}
\label{sec_conjugacy_classes}
Let $0<\lambda<1/2$ and fix $1\leq s\leq \infty$ or $s=\omega$. In this section we describe the $C^s$-conjugacy classes in the family of pseudo-affine hyperbolic IFSs of slope $\lambda$. 

For $1\leq s<\infty$ denote by $\mathcal{E}_s$ the set of pairs $(\theta_0,\theta_1)$, where $\theta_i:\mathcal{W}\to\R$ are functions satisfying:
\begin{enumerate}[(i)]
\item $0<\lambda+\theta_i(w)<1$, $\forall\ w\in\mathcal{W}$ and $i=0,1$,
\item $|\theta_i(w)|\to 0$ uniformly in $|w|\to\infty$, 
\item $\exists\ C>0$ such that $|\theta_i(w)|\leq C\cdot\left(\Psi_{(\theta_0,\theta_1)}(w)\right)^{s-1}$, $\forall\ w\in\mathcal{W}$ and $i=0,1$,
\end{enumerate} 
and $\Psi_{(\theta_0,\theta_1)}:\mathcal{W}\to(0,1)$ is as in Lemma \ref{lemma_Cantor_proportions} defined by the expression 
$$\Psi_{(\theta_0,\theta_1)}(w_1\cdots w_n)=\left(\prod_{i=1}^{n-1}(\lambda+\theta_{w_i}(w_{i+1}\cdots w_n))\right)\cdot(\lambda+\theta_{w_n}(e)),\ \forall\ w\in\mathcal{W}.$$
For $s=\infty$ denote by $\mathcal{E}_\infty$ the set of pairs satisfying (i) and (ii) above, and satisfying:
\begin{itemize}
\item[(iv)] $\forall\ s\geq 1$, $\exists\ C_s>0$ such that $|\theta_i(w)|\leq C_s\cdot\left(\Psi_{(\theta_0,\theta_1)}(w)\right)^{s-1}$, $\forall\ w\in\mathcal{W}$ and $i=0,1$.
\end{itemize}
Finally, for $s=\omega$ let $\mathcal{E}_\omega$ be the set containing only the trivial pair $(\theta_1\equiv 0,\theta_2\equiv 0)$.

\begin{defn}
For every $1\leq s\leq \infty$ or $s=\omega$, we define an equivalence relation $\sim$ on the set $\mathcal{E}_s$ by setting that $(\theta_0,\theta_1)\sim(\eta_0,\eta_1)$ if and only if there exists a continuous function $\chi:\{0,1\}^\N\to(0,+\infty)$ such that, for every sequence $a=(a_n)_{n\geq 1}$ in $\{0,1\}^\N$, we have that
\begin{equation}
\label{equation_condition-diff}
\chi(a)=\lim_{n\to+\infty}\frac{\Psi_{(\eta_0,\eta_1)}(a_1\cdots a_n)}{\Psi_{(\theta_0,\theta_1)}(a_1\cdots a_n)}.
\end{equation}
\end{defn}
It is straightforward to check that $\sim$ is indeed an equivalence relation on the set $\mathcal{E}_s$. The function $\chi$ can be interpreted as the "derivative’" of a conjugation defined along symbolic cylinders, even when the conjugation is \emph{a priori} not known to be differentiable. When the limit defining $\chi$ exists and remains bounded, it encodes the actual derivative of the conjugation, thereby characterizing the existence and  regularity  of a smooth conjugation. Formally:

\begin{thm} \label{Theorem 27}
For every $0<\lambda<1/2$ and $1\leq s\leq\infty$ or $s=\omega$, there is a one-to-one correspondence between the set of $C^s$-conjugacy classes of pseudo-affine hyperbolic IFSs of class $C^s$ with slope $=\lambda$, and the set of equivalence classes determined by $\sim$ on $\mathcal{E}_s$.
\end{thm}

\begin{proof}
Let's start by noting that the statement is trivial when $s=\omega$. As we have already remarked at Section \ref{sec_pseudo-affine_IFS}, every analytic pseudo-affine hyperbolic IFS with slope $\lambda$, coincides over the minimal invariant set with the self-similar IFS of slope $\lambda$, and therefore it has constant proportions $\lambda_i(w)=\lambda$, i.e. $\theta_i\equiv 0$, $i=0,1$.  

When $1\leq s\leq\infty$, given a hyperbolic pseudo-affine IFS $\Phi=\lbrace f_0,f_1 \rbrace$ of class $C^s$ and slope $\lambda$, write its dynamical proportions in the form $\lambda_i(w)=\lambda+\theta_i(w)$, $\forall\ w\in\mathcal{W}$, $i=0,1$. In Proposition \ref{prop_calssification_2} we showed that the correspondence $\Phi\mapsto(\theta_0,\theta_1)$ takes values in $\mathcal{E}_s$. In Proposition \ref{prop_classification_3} we showed how to construct, given any pair $(\theta_0,\theta_1)$ in $\mathcal{E}_s$, an hyperbolic pseudo-affine IFS $\Phi=\lbrace f_0,f_1 \rbrace$ of class $C^s$ with proportions $\lambda_i=\lambda+\theta_i(w)$, $\forall\ w\in\mathcal{W}$, $i=0,1$, and hence $\Phi\mapsto(\theta_0,\theta_1)$ is surjective.

Let $\Phi=\lbrace f_0,f_1 \rbrace$ and $\widetilde{\Phi}=\lbrace g_0,g_1 \rbrace$ be two pseudo-affine hyperbolic IFSs of class $C^s$ and slope $=\lambda$. Write their dynamical proportions in the form $\lambda_i^\Phi=\lambda+\theta_i$ and $\lambda_i^{\widetilde{\Phi}}=\lambda+\eta_i$, respectively. Let $X$ and $Y$ be their respective attractors, and recall from Section \ref{Section: self conformal} that, since they are binary Cantor sets, their points have a labeling $X=\{x_a:a\in\{0,1\}^\N\}$ and $Y=\{y_a:a\in\{0,1\}^\N\}$. We will denote by $\{I_w\}_{w\in\mathcal{W}}$ and $\{J_w\}_{w\in\mathcal{W}}$ the set of gaps in $[0,1]\setminus X$ and $[0,1]\setminus Y$.

As we remarked in Section \ref{Section: self conformal}, there always exist a homeomorphism $h:[0,1]\to[0,1]$ such that $h(X)=Y$ and $G_w\circ h(x)=h\circ F_w(x)$, $\forall\ x\in X$ and every finite word $w$ (here we use $G_w$ for cylinders in $Y$). Several conjugating homeomorphism $h$ can be defined, but all of them coincide over $X$ with the map taking $x_a\mapsto y_a$, $\forall a\in\{0,1\}^\N$. By means of the rigidity theorems of Bedford-Fisher \cite{BedfordFisher1997Ratio}, determining the regularity of $h:X\to Y$ reduces to determine whether $h$ is of class $C^1$ or not. Because, in the case $X$ and $Y$ are $C^s$-Hyperbolic Cantor sets with $1\leq s\leq \infty$, every $C^1$-diffeomorphism $h:X\to Y$ is actually a $C^s$-diffeomorphism. As a consequence, to conclude the theorem it is enough to prove the following claim.

\begin{cl}
\emph{
The homeomorphism $h:X\to Y$ is a $C^1$-diffeomorphism if and only if condition \eqref{equation_condition-diff} holds.
}
\end{cl}

To prove the claim, assume first that $h:X\to Y$ is a $C^1$-diffeomorphism. Without loss of generality, we can further assume that this map is the restriction onto $X$ of a $C^1$-diffeomorphism $h:[0,1]\to[0,1]$ (cf. Section \ref{sec_preliminaries}). For every sequence $a=(a_n)_{n\geq 1}$ in $\{0,1\}^\N$ we have that
\begin{align*}
\frac{\Psi_{(\eta_0,\eta_1)}(a_1\cdots a_n)}{\Psi_{(\theta_0,\theta_1)}(a_1\cdots a_n)} &= \frac{\Psi_{(\eta_0,\eta_1)}(a_1\cdots a_n)|J|}{\Psi_{(\theta_0,\theta_1)}(a_1\cdots a_n)|I|}\cdot\frac{|I|}{|J|}=\frac{|J_{a_1\cdots a_n}|}{|I_{a_1\cdots a_n}|}\cdot\frac{|I|}{|J|} =\frac{|h(I_{a_1\cdots a_n})|}{|I_{a_1\cdots a_n}|}\cdot\frac{|I|}{|J|} \\
&=\left(\frac{1}{|I_{a_1\cdots a_n}|}\int_{I_{a_1\cdots a_n}}h'(t)dt\right) \cdot\frac{|I|}{|J|}=h'(z_{a_1\cdots a_n})\cdot\frac{|I|}{|J|},
\end{align*}
where $z_{a_1\cdots a_n}$ is some point in the closure of the gap $I_{a_1\cdots a_n}$. As $n\to+\infty$, the distance between the gap $I_{a_1\cdots a_n}$ and the point $x_a\in X$ converges to zero, so in particular $z_{a_1\cdots a_n}\to x_a$ and 
$$\lim_{n\to+\infty}\frac{\Psi_{(\eta_0,\eta_1)}(a_1\cdots a_n)}{\Psi_{(\theta_0,\theta_1)}(a_1\cdots a_n)}=\lim_{n\to+\infty}h'(z_{a_1\cdots a_n})\cdot\frac{|I|}{|J|}=h'(x_a)\cdot\frac{|I|}{|J|}=\chi(a),$$
which is a continuous function.

Reciprocally, assume that \eqref{equation_condition-diff} holds and let $x_a,x_b$ be points in $X$ indexed by sequences $a,b\in\{0,1\}^\N$. To show that $h$ is differentiable over the Cantor set $X$ we have to study the limit of the incremental quotients
\begin{equation}
\label{eq_incremental_quotients}
\lim_{x_b\to x_a}\frac{h(x_b)-h(x_a)}{x_b-x_a}=\lim_{x_b\to x_a}\frac{y_b-y_a}{x_b-x_a}.
\end{equation}
Let's assume without loss of generality that $x_a$ can be accumulated from above by points $x_a<x_b$, and write 
$$\frac{y_b-y_a}{x_b-x_a}=\frac{\sum_{a<w<b}\Psi_{(\eta_0,\eta_1)}(w)|J|}{\sum_{a<w<b}\Psi_{(\theta_0,\theta_1)}(w)|I|}=\frac{\sum_{a<w<b}\Gamma(w)\cdot\Psi_{(\theta_0,\theta_1)}(w)}{\sum_{a<w<b}\Psi_{(\theta_0,\theta_1)}(w)}\cdot\frac{|J|}{|I|},$$ 
where $\Gamma(w)=\Psi_{(\eta_0,\eta_1)}(w)/\Psi_{(\theta_0,\theta_1)}(w)$. Observe that, since by hypothesis $\lim_{n\to+\infty}\Gamma(a_1\cdots a_n)=\chi(a)>0$, then given $\varepsilon>0$ there exists $N>0$ such that $\chi(a)-\varepsilon<\Gamma(a_1\cdots a_n)<\chi(a)+\varepsilon$, $\forall\ n\geq N$. On the other hand, there exists $\delta>0$ such that if $|x_b-x_a|<\delta$ then $a_i=b_i$, $\forall\ i=1,\dots,N$. This implies that all the gaps $I_w$ satisfying $x_a<I_w<x_b$ have indexation satisfying $w_i=a_i$, $\forall\ i=1,\dots,N$. It follows that
$$(\chi(a)-\varepsilon)<\frac{\sum_{a<w<b}\Gamma(w)\cdot\Psi_{(\theta_0,\theta_1)}(w)}{\sum_{a<w<b}\Psi_{(\theta_0,\theta_1)}(w)}<(\chi(a)+\varepsilon),$$
whenever $|x_b-x_a|<\delta$. This implies that for every $x_a\in X$ the limit in \eqref{eq_incremental_quotients} above exists an is equal to $\chi(a)\cdot|J|/|I|$, from where we see that $h'$ exists and is continuous over the Cantor set $X$.
\end{proof}

\section{On the proof of Theorem \ref{Theorem A}} \label{sec_proof_Thm_A}
In this section we  prove Theorem \ref{Theorem A}: For each degree of regularity $1\leq s\leq \infty$ and $0<\lambda<1/2$ we will construct a pseudo-affine hyperbolic IFS with slope $\lambda$ that is of class $C^s$ and not $C^t$, $\forall\ t>s$. Moreover, we will construct both cases (a)- $C^s$-conjugated to self-similar, and (b)- non $C^s$-conjugated to self-similar. 

\subsection{Case (a): $C^s$-conjugated to self-similar }
Using Proposition \ref{prop_classification_3} we can construct the IFS using a pair of functions $\{\theta_i:\mathcal{W}\to\R\}_{i=0,1}$ that must satisfy the criteria therein in order to have a specified degree of regularity. For obtaining IFS conjugated to self-similar, we will use functions whose values on words $w\in\mathcal{W}$ depend only on the size $n=|w|$ of the word. That is, we will define $\theta_i(w)=\varepsilon_{|w|}$ for some real sequence $(\varepsilon_n)_{n\geq 0}$. Recall from Lemma \ref{lemma_Cantor_proportions} that the lengths of the gaps $\{I_w\}_{w\in\mathcal{W}}$ associated to the Cantor set $X_\Phi$ must be given by the expression $|I_w|=\Psi(w)\cdot|I|$, $\forall\ w\in\mathcal{W}$, where $0<|I|<1$ is the length of the gap corresponding to the empty word. In this particular case the function $\Psi:\mathcal{W}\to(0,1)$ will take the form $\Psi(w)=\Psi(n)=\prod_{i=0}^{n-1}(\lambda+\varepsilon_i)$ for every word of length $|w|=n\geq 1$. 
Let's see first how to determine the sequences $(\varepsilon_n)_{n\geq 0}$ for each degree of regularity, and then the conjugation with the self-similar model.

\paragraph{Case $s=1$.}
Choose $(\varepsilon_n)_{n\geq 0}$ to be any sequence satisfying that $0<\lambda+\varepsilon_n<1$ for every $n\geq 0$, and $\varepsilon_n\to 0$ sub-exponentially. That is, $\sup\{\varepsilon_n\cdot\mu^{-n}:{n\geq 0}\}=\infty$ for all $0<\mu<1$. For instance, take the sequence $\varepsilon_0=0$ and $\varepsilon_n=\lambda n^{-\gamma}$ for $n\geq 1$, where $\gamma>0$ is some constant. Then, by Proposition \ref{prop_classification_3} there exists a pseudo-affine IFS of class $C^1$ with proportions $\lambda_i(w)=\lambda+\varepsilon_{|w|}$, $\forall\ w\in\mathcal{W}$. Using Remark \ref{rmk_exponential_convergence} above we see that this IFS cannot be of class $C^{s}$ for any $s>1$, since the functions $|\theta_i(w)|=|\varepsilon_n|$ do not converge exponentially to zero with $|w|=n\to\infty$.

\paragraph{Case $1<s<\infty$.}
Write $s=r+\alpha$ with $r\in\mathbb{N}$ and $0<\alpha\leq 1$.  It is enough to show the following:
\begin{cl}\label{cl_sequence_epsilon_n_case_S>1}
\emph{
There exists a sequence $(\varepsilon_n)_{n\geq 0}$ such that
\begin{enumerate}[a.]
\item $\varepsilon_n\to 0$ with $n\to+\infty$ and $0<\lambda+\varepsilon_{n}<1$, $\forall\ n\geq 0$;
\item $|\epsilon_{n}|=\Psi(n)^{s-1}$, $\forall\ n\geq 1$;
\end{enumerate}
}
\end{cl}
By Proposition \ref{prop_classification_3}, item a. above implies the existence of a pseudo-affine IFS with dynamical proportions $\lambda_i(w)=\lambda+\theta_i(w)$, where $\theta_i(w)=\varepsilon_{|w|}$, $\forall\ w\in\mathcal{W}$ and $i=0,1$. Item b. above implies that $|\theta_i(w)|=D\cdot|I_w|^{s-1}$ for the constant $D=(1/|I|)^{s-1}>0$. Hence, the IFS is of class $C^{r,\alpha}$. Finally, using Proposition \ref{prop_calssification_2} we can check that the IFS is not $C^t$ for $t>s$, since 
$$\sup_{w\in\mathcal{W}}\left\{\frac{|\theta_i(w)|}{|I_w|^{t-1}}\right\}=\sup_{w\in\mathcal{W}}\left\{\frac{D}{|I_w|^{t-s}}\right\}=+\infty,\ \text{whenever}\ t>s.$$

\begin{proof}[Proof of Claim \ref{cl_sequence_epsilon_n_case_S>1}]
Let $\varepsilon_0=0$, $\varepsilon_1=\lambda^{s-1}$, and for $n\geq 1$ define $\varepsilon_{n+1}=T_{(\lambda,s-1)}(\varepsilon_n)$, where the map $T_{(\lambda,s-1)}:\R\to\R$ is given by $T_{(\lambda,s-1)}(u)=u\cdot(\lambda+u)^{s-1}$. 
To check item b. proceed by induction. For $n=1$ then $\varepsilon_1=(\lambda+\varepsilon_0)^{s-1}=\Psi(1)^{s-1}$ is verified, and assuming $\varepsilon_n=\Psi(n)^{s-1}$ for some $n\geq 1$ then 
$$\varepsilon_{n+1}=T_{(\lambda,s-1)}(\varepsilon_{n})=\varepsilon_{n}(\lambda+\varepsilon_n)^{s-1}=\Psi(n)^{s-1}(\lambda+\varepsilon_n)^{s-1}
=\left(\prod_{i=0}^{n-1}(\lambda+\varepsilon_i)^{s-1}\right)\cdot(\lambda+\varepsilon_n)^{s-1}=\Psi(n)^{s-1}.$$
To check item a. observe that for every $0<\lambda<1/2$ and $s>1$ we have the strict inequality $T_{(\lambda,s-1)}(u)<u$ in the open interval $0<u<1/2$. Thus, every sequence $x_{n}=T_{(\lambda,s-1)}(x_0)$ satisfies that $0<x_n<1/2$ and $x_n\to 0$, provided $0<x_0<1/2$. Applying this to the sequence $(\varepsilon_n)_{n\geq 0}$ we have that $0<\lambda+\varepsilon_n<1$ for $n=0,1$, and since $0<\varepsilon_1<1/2$ then $\varepsilon_n\to 0$ and $0<\lambda+\varepsilon_n<1$ for every $n\geq 2$. This proves the claim. 
\end{proof}

\paragraph{Case $s=\infty$.}
This case follows the same argument as before.
\begin{cl}\label{cl_sequence_epsilon_n_case_S=infty}
\emph{
There exists a sequence $(\varepsilon_n)_{n\geq 0}$ such that
\begin{enumerate}[a.]
\item $\varepsilon_n\to 0$ with $n\to+\infty$ and $0<\lambda+\varepsilon_{n}<1$, $\forall\ n\geq 0$;
\item $|\epsilon_{n}|=\Psi(n)^{n}$, $\forall\ n\geq 1$;
\end{enumerate}
}
\end{cl}
Taking $\theta_i(w)=\varepsilon_{|w|}$ the statement in Proposition \ref{prop_classification_3} gives again pseudo-affine IFS, that in this case is of class $C^\infty$ since  
$$\sup_{w\in\mathcal{W}}\left\{\frac{|\theta_i(w)|}{|I_w|^{k}}\right\}=\sup_{w\in\mathcal{W}}\left\{\frac{1}{|I|^k}\cdot\frac{\Psi(|w|)^{|w|}}{\Psi(|w|)^{k}}\right\}=0,\ \text{for every}\ k\geq 1.$$
The IFS is not analytic, because the functions $\theta_i(w)$ are not constant $=0$. 

\begin{proof}[Proof of Claim \ref{cl_sequence_epsilon_n_case_S=infty}]
Let $\varepsilon_0=0$, $\varepsilon_1=\lambda$, and for $n\geq 1$ define $\varepsilon_{n+1}=T_{(\lambda,n)}(\varepsilon_n)$, where in this case the map $T_{(\lambda,n)}:\R\to\R$ is given by $T_{(\lambda,n)}(u)=u^{\frac{n+1}{n}}\cdot(\lambda+u)^{n+1}$. To check item b. we have
$$\varepsilon_{n+1}=T_{(\lambda,n)}(\varepsilon_{n})=\varepsilon_{n}^{\frac{n+1}{n}}(\lambda+\varepsilon_n)^{n+1}=\Psi(n)^{n+1}(\lambda+\varepsilon_n)^{n+1}
=\left(\prod_{i=0}^{n-1}(\lambda+\varepsilon_i)^{n+1}\right)\cdot(\lambda+\varepsilon_n)^{n+1}=\Psi(n)^{n+1}.$$
To check item a. observe that for every $0<\lambda<1/2$ and $n\geq 1$, on the open interval $0\leq u\leq 1/2$ we have the inequality $T_{(\lambda,n)}(u)<u(1/2+u)$, and hence $0<\varepsilon_{n+1}=T_{(\lambda,n)}(\varepsilon_n)<1/2$, $\forall n\geq 1$ and $\varepsilon_n\to 0$ when $n\to+\infty$, since $0<\varepsilon_1=\lambda<1/2$.
\end{proof}

To check conjugation with the self-similar IFS we will use Theorem \ref{Theorem 27}. Denote by $\Phi$ the IFS with proportions $\lambda_i^\Phi(w)=\lambda+\varepsilon_{|w|}$, for the sequences $(\varepsilon_n)_{n\geq 0}$ constructed above, and denote by $\widetilde{\Phi}$ the self-similar IFS of slope $\lambda$, that has constant proportions $\lambda_i^{\widetilde{\Phi}}(w)=\lambda$. Let $a=(a_n)_{n\geq 0}$ be a sequence in $\{0,1\}^\N$. Then, for every $n\geq 0$ the expression \eqref{equation_condition-diff} from Theorem \ref{thm_B_conjugation} takes the form:
\begin{equation*}
\chi(a_1\cdots a_n)=\frac{\Psi_{(\varepsilon,\varepsilon)}(a_1\cdots a_n)}{\Psi_{(0,0)}(a_1\cdots a_n)}=\prod_{i=0}^{n-1}\left(1+\frac{\varepsilon_i}{\lambda}\right).
\end{equation*}
Since the right hand side in the equality above only depends on $n=|w|$ and not in the particular sequence $a$, we see that the limit 
$\lim_{n\to\infty}\chi(a_1\cdots a_n)$ exists if $\sum_{i=0}^\infty|\varepsilon_i|<\infty$, and it is constant in case of existence. Therefore, in the case of regularity $1<s\leq\infty$, since $\varepsilon_n\to 0$ exponentially, we obtain that $\chi:\{0,1\}^\N\to\R_+$ is a constant function, and thus $\Phi$ is $C^s$-conjugated to $\widetilde{\Phi}$. In the case of regularity $s=1$, the pseudo-affine IFS $\Phi$ is $C^1$-conjugated to $\widetilde{\Phi}$ only when $\varepsilon_n$ is summable, e.g. when $\varepsilon_n\simeq \lambda n^{-\gamma}$, for some constant $\gamma>1$. 

\begin{rem}
The hyperbolic Cantor sets $X_\Phi$ obtained above have constant scaling function $\rho=(\lambda,1-2\lambda,\lambda)$. This follows from the fact that they are $C^s$-conjugated to self-similar, but it may also be checked directly from the definition of scaling function. Indeed, since all the gaps at stage $n\geq 1$ in the construction of the binary Cantor set have the same size, then all the cylinders of the Cantor set at level $n$ have the same size. It follows that the geometric ratios $r_i(a_n\cdots a_1)$ (cf. Section \ref{sec_preliminaries}) only depend on $n$, and thus taking limit when $n\to+\infty$ we obtain constant scaling functions.    
\end{rem}

\subsection{Case (b): non-$C^s$-conjugated to self-similar}

\begin{figure}[t]
\begin{center}
\includegraphics[width=\textwidth]{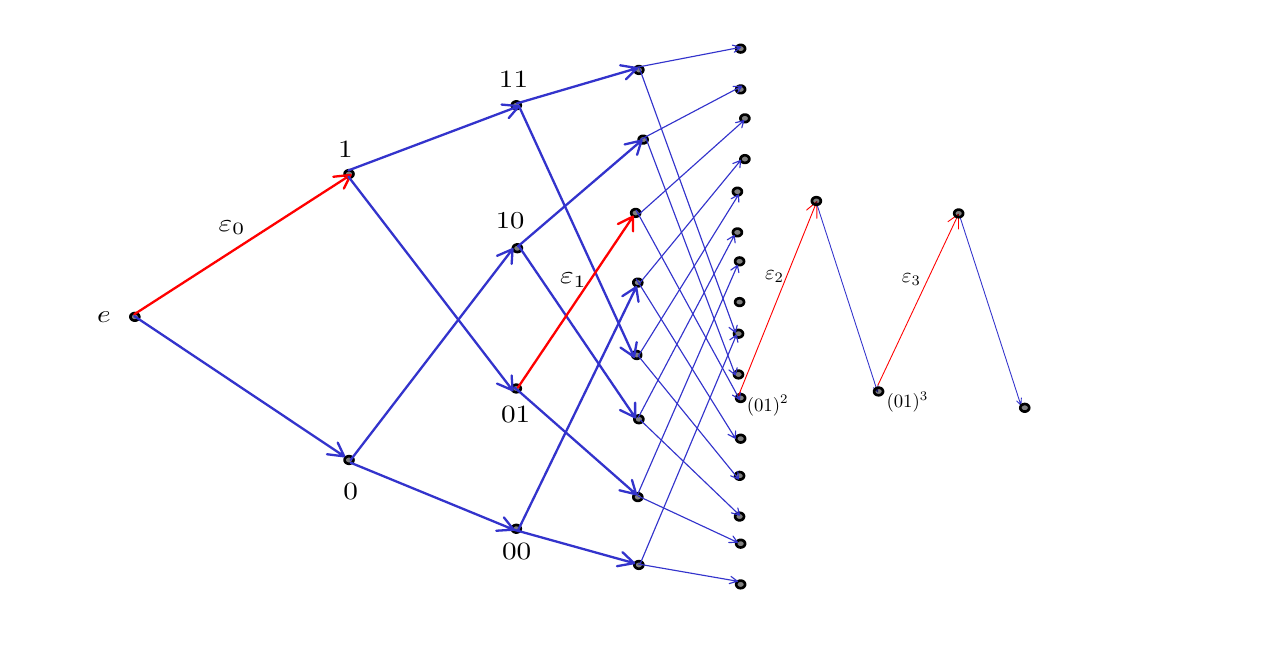}
\caption{Edges in blue correspond to $\theta_i(w)=0$, otherwise colored in red.}
\label{fig_non-conjugated_example}
\end{center}
\end{figure}

\paragraph{Case $s=1$.}
Let's start by observing that, in the case $s=1$, the examples from the previous subsection are not conjugated to self-similar, provided the sequence $(\varepsilon_n)_{n\geq 0}$ is not summable. For example, if we choose $\varepsilon_n\simeq \lambda/n$ in the setting above, we obtain a pseudo-affine hyperbolic IFS of class $C^1$, that is not $C^1$-conjugated to self-similar.

For the case $1<s\leq\infty$ we will be forced to consider a pair of functions $\{\theta_i:\mathcal{W}\to\R\}_{i=0,1}$ that depend non trivially on the branches $i=0,1$ of the IFS. For this, we will choose suitable sequences $(\varepsilon_k)_{k\geq 0}$ of positive reals and define the functions $\theta_i(w)$ in the following way: For the empty word $e\in\mathcal{W}$ we set $\theta_0(e)=0$, $\theta_1(e)=\varepsilon_0$, and for every finite word $w=w_1\cdots w_n$, $n\geq 1$ we set
\begin{align}
\label{eq_theta_i-non-symmetric}
& \theta_i(w_1\cdots w_n) = \left\lbrace
\begin{matrix}
&\varepsilon_k & ;\ \text{if}\ \ n=2k,\ w=(01)^k,\ \text{and}\ i=1;\\
& 0       & ;\ \text{otherwise}.
\end{matrix}
\right.
\end{align}
The corresponding proportions $\lambda_i(w)=\lambda+\theta_i(w)$ are represented in the graph at Figure \eqref{fig_non-conjugated_example}. We explain first how to choose the sequences $(\varepsilon_k)_{k\geq 0}$ for each degree of regularity, and we analyze the conjugation problem at the end.

\paragraph{Case $1<s<\infty$.}
The construction relies on the following claim:
\begin{cl}\label{cl_non-conjugated_case_s>1}
\emph{
There exists a sequence $(\varepsilon_k)_{k\geq 0}$ of positive reals such that
\begin{enumerate}[a.]
\item $\varepsilon_k\to 0$ with $k\to+\infty$ and $0<\lambda+\varepsilon_{k}<1$, $\forall\ k\geq 0$;
\item $\varepsilon_{k}=\left(\lambda^k\cdot\prod_{r=0}^{k-1}(\lambda+\varepsilon_{r})\right)^{s-1}$, $\forall\ k\geq 1$;
\end{enumerate}
}
\end{cl}

Assuming this claim, we have to check that the $\theta_i(w)$ given in expression \eqref{eq_theta_i-non-symmetric} satisfy the hypothesis of Proposition \ref{prop_classification_3}. Since $\theta_i(w)$ is either zero or equal to some $\varepsilon_k$, by item a. we have $0<\lambda+\theta_i(w)<1$, $\forall w\in\mathcal{W}$. Using item b. for every integer $k\geq 0$ we have 
\begin{align*}
\theta_1\left((01)^k\right)=\varepsilon_{k}
&=\left(\lambda^k\cdot\prod_{r=0}^{k-1}(\lambda+\varepsilon_{r})\right)^{s-1}
=\left(\prod_{r=0}^{k-1}(\lambda+0)(\lambda+\varepsilon_{k-r-1})\right)^{s-1}\\
&=\left(\prod_{r=0}^{k-1}\left(\lambda+\theta_{0}\left(1(01)^{k-r-1}\right)\right)\cdot\left(\lambda+\theta_{1}\left((01)^{k-r-1}\right)\right)\right)^{s-1}\\
&=\left(\prod_{l=1}^{2k}(\lambda+\theta_{w_l}(w_{l+1}\cdots w_{2k}))\right)^{s-1}=\left(\Psi_{(\theta_1,\theta_2)}\left((01)^k\right)\right)^{s-1}
\end{align*}
Since $\theta_i(w)=0$ for all other choices of $i$ and $w$, we see that $|\theta_i(w)|\leq \cdot(\Psi_{(\theta_1,\theta_2)}(w))^{s-1}$, $\forall$ $w\in\mathcal{W}$ and $i=0,1$. Hence, by Proposition \ref{prop_classification_3} we conclude the existence of a pseudo-affine IFS of class $C^s$ and dynamical proportions $\lambda_i(w)=\lambda+\theta_i(w)$. Finally, using Proposition \ref{prop_calssification_2} we can check that the IFS is not $C^t$ for $t>s$. Observe that for the word $w=(01)^k$, $k\geq 0$ it is satisfied that $|\theta_1((01)^k)|=D\cdot|I_{(01)^k}|^{s-1}$ for the constant $D=(1/|I|)^{s-1}>0$. Hence 
$$\sup_{w\in\mathcal{W}}\left\{\frac{|\theta_i(w)|}{|I_w|^{t-1}}\right\}\geq \sup_{k\geq 0}\left\{\frac{D}{|I_{(01)^k}|^{t-s}}\right\}=+\infty,\ \text{whenever}\ t>s.$$

\begin{proof}[Proof of Claim \ref{cl_non-conjugated_case_s>1}]
Choose any $0<\varepsilon_0<1/2$, so $0<\lambda+\varepsilon_0<1$. Define $\varepsilon_1= \left( \lambda\cdot(\lambda+\varepsilon_0) \right)^{s-1}$, which necessarily satisfies that $0<\lambda+\varepsilon_1<1$. For $k\geq 1$ define $\varepsilon_{k+1}=T_{(\lambda,s-1)}(\varepsilon_k)$, where $T_{(\lambda,s-1)}(u)=\lambda^{s-1}u(\lambda+u)^{s-1}$. To check item b. we have that
\begin{align*}
\varepsilon_{k+1}
=T_{(\lambda,s-1)}(\varepsilon_{k})=\lambda^{s-1}\varepsilon_{k}(\lambda+\varepsilon_k)^{s-1}
&=\lambda^{s-1}\left(\lambda^{k}\prod_{r=0}^{k-1}(\lambda+\varepsilon_r)\right)^{s-1}(\lambda+\varepsilon_k)^{s-1}
=\left(\lambda^{k+1}\prod_{r=0}^{k}(\lambda+\varepsilon_r)\right)^{s-1}.
\end{align*}
Finally, since $T_{(\lambda,s-1)}(u)<u$ in the open interval $0<u<1/2$ and $0<\varepsilon_1<1/2$, we conclude that $0<\lambda+\varepsilon_k<1$, $\forall k\geq 0$ and $\varepsilon_k\to 0$ when $k\to+\infty$. 
\end{proof}

\paragraph{Case $s=\infty$.}
The argument in this case is analogous to the previous one. 
\begin{cl}\label{cl_non-conjugated_case_s=infty}
\emph{
There exists a sequence $(\varepsilon_k)_{k\geq 0}$ of positive reals such that
\begin{enumerate}[a.]
\item $\varepsilon_k\to 0$ with $k\to+\infty$ and $0<\lambda+\varepsilon_{k}<1$, $\forall\ k\geq 0$;
\item $\varepsilon_{k}=\left(\lambda^k\cdot\prod_{r=0}^{k-1}(\lambda+\varepsilon_{r})\right)^{k}$, $\forall\ k\geq 1$;
\end{enumerate}
}
\end{cl}
By means of Proposition \ref{prop_classification_3}, taking $\theta_i(w)$ as defined in \eqref{eq_theta_i-non-symmetric} gives a pseudo-affine IFS, that in this case is of class $C^\infty$ since  
$$\sup_{w\in\mathcal{W}}\left\{\frac{|\theta_i(w)|}{|I_w|^{k}}\right\}=\sup_{m\geq 0}\left\{\frac{1}{|I|^k}\cdot\frac{\Psi((01)^m)^{m}}{\Psi((01)^m)^{k}}\right\}=0,\ \text{for every}\ k\geq 1.$$
The IFS is not analytic, because the functions $\theta_i(w)$ are not constant $=0$. 

\begin{proof}[Proof of Claim \ref{cl_non-conjugated_case_s=infty}]
Let $0<\varepsilon_0<1/2$ and $\varepsilon_1=\lambda(\lambda+\varepsilon_0)$, which satisfy $0<\lambda+\varepsilon_k<1$ for $k=0,1$. For every $k\geq 1$ define $\varepsilon_{k+1}=T_{(\lambda,k)}(\varepsilon_k)$, where $T_{(\lambda,k)}(u)=\lambda^{k+1}u^{\frac{k+1}{k}}\cdot(\lambda+u)^{k+1}$. To check item b. we have
\begin{align*}
\varepsilon_{k+1}
=T_{(\lambda,k)}(\varepsilon_{k})=\lambda^{k+1}\varepsilon_{k}^{\frac{k+1}{k}}(\lambda+\varepsilon_k)^{k+1}
&=\lambda^{k+1}\left(\lambda^{k}\prod_{r=0}^{k-1}(\lambda+\varepsilon_r)\right)^{k+1}(\lambda+\varepsilon_k)^{k+1}
=\left(\lambda^{k+1}\prod_{r=0}^{k}(\lambda+\varepsilon_r)\right)^{k+1}.
\end{align*}
To check item a. observe that for every $0<\lambda<1/2$ and $k\geq 1$, on the open interval $0\leq u\leq 1/2$ we have $T_{(\lambda,k)}(u)<u(1/2+u)$, and hence $0<\varepsilon_{k+1}=T_{(\lambda,n)}(\varepsilon_k)<1/2$, $\forall k\geq 1$ and $\varepsilon_k\to 0$ when $k\to+\infty$, since $0<\varepsilon_1<1/2$.
\end{proof}

Finally, using Theorem \ref{thm_B_conjugation}, we will show that the IFSs $\Phi$ constructed above are not $C^s$-conjugated to the self-similar IFS $\widetilde{\Phi}$, for every $1<s\leq\infty$. For this, we will show the non-existence of the limit $\lim_{n \to\infty}\chi(a_1\cdots a_n)$ given in the expression \eqref{equation_condition-diff} on the word $a=(01)^\infty$. For every $k\geq 0$ this expression can be written as
\begin{align*}
&\chi(a_1\cdots a_{2k})= \chi\left((01)^k\right)=
\prod_{r=0}^{k-1}\underbrace{\left(1+\frac{1}{\lambda}\theta_{0}\left(1(01)^{k-r-1}\right)\right)}_{1}\cdot\underbrace{\left(1+\frac{1}{\lambda}\theta_{1}\left((01)^{k-r-1}\right)\right)}_{1+\frac{1}{\lambda}\varepsilon_{k-r-1}}\geq (1+\frac{1}{\lambda}\varepsilon_0)>1;\\
&\chi(a_1\cdots a_{2k+1})= \chi\left((01)^k 0\right)=
\prod_{r=0}^{k-1}\underbrace{\left(1+\frac{1}{\lambda}\theta_{1}\left(1(01)^{k-r-1}\right)\right)}_{1}\cdot\underbrace{\left(1+\frac{1}{\lambda}\theta_{0}\left((01)^{k-r-1}\right)\right)}_{1}=1. 
\end{align*}
Thus, the sequence $\chi(a_1\cdots a_n)$ oscillates and has no limit, which implies that $\Phi$ cannot be $C^s$-conjugated to the self-similar IFS $\widetilde{\Phi}$. 


\section{On the proof of Theorem \ref{Threom C}}

\subsection{On the proof of Theorem \ref{Threom C} Part (2)}
\label{sec_Livshitz-I}

Let $\widetilde{\Phi}= \lbrace g_0,g_1 \rbrace$ be a hyperbolic IFS  of class $C^s$, $s\geq 1$, and let $Y$ be its attractor. Recall that for each finite word $w\in\mathcal{W}$ we denote by $y_w$ the fixed point of the contraction $G_w:[0,1]\to[0,1]$. In this section we investigate whether $\widetilde{\Phi}$ is $C^s$-conjugated to some hyperbolic pseudo-affine IFS $\Phi= \lbrace f_0,f_1 \rbrace$. 

In this Section we prove Theorem \ref{Threom C} Part (2). Let us recall its statement:
\begin{thm}
\label{thm_B_conjugation}
Suppose $\widetilde{\Phi}$ is of class $C^{r,\alpha}$ for some $r\geq 1$ and $0<\alpha\leq 1$. If there exists $0<\lambda<1$ such that
\begin{equation}
\label{livshitz_condition}
G_w'(y_w)=\lambda^{|w|}, \forall w\in\mathcal{W},
\end{equation}
then $\widetilde{\Phi}$ is $C^{r,\alpha}$-conjugated to a hyperbolic pseudo-affine IFS $\Phi= \lbrace f_0,f_1 \rbrace$ of class $C^{r,\alpha}$, with slope $=\lambda.$
\end{thm}

The proof of Theorem \ref{thm_B_conjugation} relies on the following Proposition.

\begin{prop}
\label{prop_fundamental_conjugation}
Let $\Phi= \lbrace f_0,f_1 \rbrace$ be a hyperbolic IFS  of class $C^1$, and let $X$ be its attractor. If there exists $0<\lambda<1$ such that:
\begin{enumerate}[(i)]
\item Livsic condition: $F_w'(x_w)=\lambda^{|w|}$, $\forall\ w\in\mathcal{W}$, where $x_w=F_w(x_w)$; and
\item $f_0'(x)=\lambda$, $\forall x\in X_F$. 
\end{enumerate}
Then $f_1'(x)=\lambda$, $\forall x\in X$. That is, the IFS is pseudo-affine of slope $\lambda$.
\end{prop}
We remark that the following proofs are related to the arguments given in \cite[Section 6]{algom2023polynomial}.
\begin{proof}[Proof of Theorem \ref{thm_B_conjugation} assuming Proposition \ref{prop_fundamental_conjugation}]
If we consider the action on the interval generated only by the contraction $g_0:[0,1]\to[0,1]$, under the hypothesis of regularity $s=r+\alpha>1$, the Poincar\'{e}-Siegel Theorem \cite[Theorem 2.8.2]{Katok1995Hass} asserts that $g_0$ is $C^s$-conjugated to an affine contraction. That is, there exists a $C^s$-diffeomorphism $h:[0,1]\to[0,1]$ such that $E_\lambda\circ h=h\circ g_0$, where $E_\lambda:t\mapsto\lambda\cdot t$ and $\lambda=g_0'(0)$.

Consider the IFS $\Phi= \lbrace f_0,f_1 \rbrace$ defined by $f_0=E_\lambda$ and $f_1=h\circ g_1\circ h^{-1}$, which is of class $C^s$. It is direct to see that $F_w\circ h=h\circ G_w$, for every finite word $w\in\mathcal{W}$. Applying the chain rule we get that 
$$F_w'(x_w)=G_w'(y_w)=\lambda^{|w|},\ \forall\ w\in\mathcal{W},\ \text{where}\ x_w=h(y_w).$$
Hence, the IFS $\Phi$ satisfies the hypothesis of Proposition \ref{prop_fundamental_conjugation}. Thus, it is pseudo-affine, as required.  
\end{proof}

\begin{proof}[Proof of Proposition \ref{prop_fundamental_conjugation}]
For every finite word $w=w_1\cdots w_n$ let $p_w\in X$ be the point whose coding in $\{0,1\}^\N$ is the sequence starting with $w$ followed by infinitely many zeros. For the empty word this corresponds, by our definition of an hyperbolic IFS, to $p_e=0\in[0,1]$. For any other word $w$ finishing with $w_n=1$ the point $p_w$ is the left boundary point of the closed interval $K_w=F_w([0,1])$ (cf. Section \ref{Section: binary cantor set}). Since the set $\{p_w:w\in\mathcal{W}\}$ is dense in $X$, to prove the Proposition it is enough to show that $f_1'(p_w)=\lambda$, for every $w\in\mathcal{W}$. We will proceed by induction on the size $n=|w|$ of the words.

\smallskip
\noindent
\textbf{Base case:} \emph{We have that $f_1'(0)=\lambda$.}
\smallskip

Observe that from the Livsic condition $(i)$ we directly obtain that $f_1'(1)=\lambda$, but we do not have a priori information about $f_1'(0)$. To prove this, we will approximate the point $p_e=0$ with some special periodic points where we can calculate the derivative of $f_1$. 

For every $m\geq 0$ let $x_m\in X$ be the point with coding
$ (\overbrace{0\cdots 0}^{m}1)^{\infty}$. That is, the coding of $x_m$ is the periodic sequence obtained by infinite concatenation with itself of the string consisting in $m$ consecutive symbols equal to $0$ followed by one $1$. We have that
\begin{equation*}
\lambda^{m+1}=F_{(0\cdots 0 1)}'(x_m)=\left(\prod_{k=0}^{m-1}f_0'\left(f_0^k\circ f_1(x_w)\right)\right)\cdot f_1'(x_m)=\lambda^{m}\cdot f_1'(x_m),
\end{equation*} 
from where it follows that $f_1'(x_m)=\lambda$. Since $x_m\to 0$ with $m\to+\infty$ and $f_1'$ is continuous, we conclude that $f_1'(0)=\lambda$.

\smallskip
\noindent
\textbf{Inductive step:} 
\emph{Given $n\geq 1$, assume that $f_1'(p_u)=\lambda$ for every word $u\in\mathcal{W}$ of length $|u|<n$. Then, it is verified that $f_1'(p_w)=\lambda$ for every word $w$ of length $|w|=n$.}
\smallskip

Let $w=w_1\cdots w_n$ be a finite word and let $x_m\in X_F$ be the point with coding $x_m\simeq w\cdot(\overbrace{0\cdots 0}^{m}1w)^\infty$. Since $F_{(w 0\cdots 01)}(x_m)=x_m$ there is a closed orbit
\begin{align*}
x_m\xmapsto[]{f_1} x_m^1\xmapsto[]{f_0^m} z_m\xmapsto[]{f_{w_n}} z_m^1 \mapsto\cdots\mapsto z_m^{n-1}\xmapsto[]{f_{w_1}} x_m,
\end{align*}
where $x_m^1=f_1(x_m)$, $z_m  =f_0^m(x_m^1)$ and $z_m^k =f_{w_{n+1-k}}(z_m^{k-1})$, $\forall\ k=1,\cdots,n-1$. Applying the chain rule on this expression we obtain  
\begin{align*}
\lambda^{n+m+1}=\left(\prod_{i=1}^n f_{w_i}'(z_m^{n-i})\right)\cdot\lambda^m\cdot f_1'(x_m).
\end{align*}
and hence 
\begin{equation}
\label{eq_estimate_Df_1}
\frac{f_1'(x_m)}{\lambda}=\left(\prod_{i=1}^n\frac{f_{w_i}'(z_m^{n-i})}{\lambda}\right)^{-1}.
\end{equation}

We can estimate now the value of each $f_{w_i}'(z_m^{k})$, $k=0,\dots,n-1$, by applying the inductive hypothesis. For this, observe that the coding of the points $z_m^k$ are given by the expressions
\begin{align*}
z_m &=(\underbrace{0\cdots 0}_{m}1w)^\infty, \\
z_m^1 &=w_n\cdot(\underbrace{0\cdots 0}_{m}1w)^\infty, \\
& \vdots \\
z_m^{n-1} &=w_2\cdots w_n\cdot(\underbrace{0\cdots 0}_{m}1w)^\infty.
\end{align*}
Define $q_0=0$ and for each $k=1,\dots n-1$ define $q_k\in X$ to be point $q_k=p_{w_{n+1-k}\cdots w_n}$ whose coding is $q_k\simeq w_{n+1-k}\cdots w_n \cdot (0)^\infty$. Observe that for each $k=0,\dots,n-1$ the coding of $z_m^{k}$ and $q_k$ coincide at least up to order $m$. This means that $z_m^k\to q_k$ when $m\to+\infty$, and since $f_1'$ is uniformly continuous then also $f_1'(z_m^k)\to \lambda$ when $m\to+\infty$. Using this together with the expression in equation \eqref{eq_estimate_Df_1} above, we obtain that
$$\lim_{m\to+\infty}\left(\prod_{i=1}^n\frac{f_{w_i}'(z_m^{n-i})}{\lambda}\right)^{-1}=1,$$
and since $x_m\to p_w$ when $m\to+\infty$ we conclude that $f_1'(p_w)=\lambda$. This concludes the inductive step, and the proof of the Proposition.
\end{proof}

\begin{rem}
It is interesting to note that in this proof we do not actually need the hypothesis of $C^s$-regularity with $s>1$ for both branches of the IFS. Instead, it is enough to have one branch $f_i$ that is $C^1$-conjugated to an affine contraction.   
\end{rem}


\subsection{On the proof of Theorem \ref{Threom C} Part (1)} \label{sec_Livshitz-II}
In this section we prove Theorem \ref{Threom C} Part (1). Namely, we show that for a $C^1$ hyperbolic IFS, the Livsic condition
\eqref{livshitz_condition} does not ensure $C^1$-conjugacy to a pseudo-affine IFS.
 Formally,
\begin{thm} \label{Theorem livsic fails}
There exists a $C^1([0,1])$ hyperbolic IFS $\Phi= \lbrace g_0,g_1 \rbrace\subseteq C^1([0,1])$ such that:
\begin{enumerate}
\item It satisfies the Livsic condition: for some $\lambda \in (0,\frac{1}{2})$,
$$G_w'(y_w)=\lambda^{|w|}, \forall w\in\mathcal{W}, $$
\item It is not $C^1$-conjugate to a pseudo-affine IFS. 
\end{enumerate}

\end{thm}

\medskip
We begin by setting our framework: Let $T:[0,1]\to [0,1]$ be the {\em angle-tripling map} 
$$T(x)=3\cdot x\text{ mod }1,$$ 
and let $\Lambda\subseteq [0,1]$ denote the ($T$ invariant) middle thirds Cantor set. 

\begin{thm}\label{mainThm}
There exist an increasing homeomorphism $h:[0,1]\to [0,1]$ and a continuous function $\varphi:\Lambda \rightarrow \mathbb{R}$ such that:
\begin{enumerate}
\item $\varphi$ is not a $T|_\Lambda$ coboundary: There does \textbf{not} exist a continuous function $u:\Lambda\rightarrow \mathbb{R}$ such that
$$\varphi=u-u\circ T|_\Lambda.$$
Nonetheless, $\int_\Lambda \varphi\, d\nu = 0$ for every $T|_\Lambda$--invariant
Borel probability measure $\nu$ on $\Lambda$.

\item Writing $\widehat{T}:=h^{-1}\circ T\circ h$, we have
$\widehat{T}\in C^1([0,1])$, $\widehat{T}(h^{-1}(\Lambda))=h^{-1}(\Lambda)$,
and there exists some $\log 2<P<\log 3$ such that for every $x\in h^{-1}(\Lambda)$,
$$
\log \widehat{T}'(x)=P-\varphi(h(x))>0.
$$
\end{enumerate}	
\end{thm}
We proceed to deduce Theorem \ref{Theorem livsic fails} from Theorem \ref{mainThm}. Set $\widehat T:=h^{-1}\circ T\circ h$ and  $\widehat\Lambda:=h^{-1}(\Lambda)$ as in Theorem \ref{mainThm}. We will show two statements: First,
that the restriction $\widehat T|_{\widehat\Lambda}$ satisfies a 
Livsic 
condition:
\begin{equation} \label{eq expandin livsic}
(\widehat T^n)'(p)=e^{nP}, \qquad \text{for all } p\in \mathrm{Per}_n(\widehat T),
\end{equation}
where $\mathrm{Per}_n(\widehat T) = \lbrace p\in \widehat\Lambda:\, \widehat T^n p =p \rbrace$ denotes the set of $n$-periodic points. Second, we will show that $\widehat T$ is not $C^1$-conjugate to any expanding map on a compact invariant set that has constant derivative on that set.

Consequently, taking the inverse branches of $\widehat T$, we obtain a $C^1([0,1])$ hyperbolic IFS satisfying
the 
Livsic condition \eqref{livshitz_condition}, but not $C^1$-conjugate to a pseudo--affine IFS. Note that for this purpose the fact that $P>\log 2$ is only used to ensure that the derivatives of both branches are bounded by $e^{-P}<\frac{1}{2}$. Since $h$ is a homeomorphism conjugating it to the middle-$1/3$ Cantor set, the first branch fixes $0$ and the second fixes $1$, whence this IFS is indeed hyperbolic. This is Theorem~\ref{Theorem livsic fails}.

We thus begin by establishing \eqref{eq expandin livsic}. By Theorem \ref{mainThm} Part (2),
\[
\log \widehat T'(x)=P-\varphi(h(x))
\quad \text{for all } x\in\widehat\Lambda .
\]
If $p\in\mathrm{Per}_n(\widehat T)$, then $h(p)\in\Lambda$
is $n$-periodic for $T$. So, by summing along its orbit, we get
\[
\log (\widehat T^n)'(p)
=\sum_{k=0}^{n-1}\log \widehat T'(\widehat T^k p)
= n\cdot P-\sum_{k=0}^{n-1}\varphi(T^k h(p)).
\]
By Theorem~\ref{mainThm} Part (1), the function $\varphi$
has zero integral against every $T|_\Lambda$-invariant probability measure.
In particular,  $q:=h(p)$ has period $n$, so the
periodic orbit measure
\[
\nu_q:=\frac1n\sum_{k=0}^{n-1}\delta_{T^k q}
\]
is $T$-invariant. Therefore
\[
0=\int \varphi\,d\nu_q
=\frac1n\sum_{k=0}^{n-1}\varphi(T^k q).
\]
We thus obtain
\[
\log (\widehat T^n)'(p)=n\cdot P,
\]
which proves \eqref{eq expandin livsic}.

Now, assume for contradiction that $\widehat T=g^{-1}\circ A\circ g$
where $g\in C^1([0,1])$ is a diffeomorphism,  $A\in C^1([0,1])$  preserves and is expanding on the compact set $\Lambda_A$, and furthermore has constant derivative $a\in \mathbb{R}$ on $\Lambda_A$. In particular, $g'$ is non vanishing on $\widehat{\Lambda}$, and so $\log |g'|$ is continuous on $\widehat{\Lambda}$.

We first claim that $a=e^P$. Since $0\in\Lambda$ and $T(0)=0$, the map
$T|_\Lambda$ has a fixed point, hence $\widehat T|_{\widehat\Lambda}$ has
the fixed point $p:=h^{-1}(0)\in\widehat\Lambda$.
Let $q:=g(p)\in\Lambda_A$. Then $A(q)=q$, and differentiating
$\widehat T=g^{-1}\circ A\circ g$ at $p$ gives
\[
\widehat T'(p)= (g^{-1})'(A(g(p)))\cdot A'(g(p))\cdot g'(p)
= \frac{A'(q)\,g'(p)}{g'(\widehat T(p))}.
\]
Since $p$ is a fixed point of $\widehat T$, we have $\widehat T(p)=p$, hence
\[
\widehat T'(p)=A'(q)=a.
\]
On the other hand, by \eqref{eq expandin livsic} with $n=1$ we have
$\widehat T'(p)=e^P$. Hence $a=e^P$.

We therefore obtain on $\widehat\Lambda=g^{-1}(\Lambda_A)$ that
\[
\log \widehat T'= \log|\widehat T'|
=\log|g'|-\log|g'|\circ \widehat T+\log|a|.
\]
Combining this with Theorem~\ref{mainThm} Part (2) (noting that
$\widehat T'(x)>0$ for $x\in\widehat\Lambda$) and with $|a|=e^P$ yields
\[
\varphi\circ h
= -\log|g'| + \log|g'|\circ \widehat T .
\]
Setting $u:=-(\log|g'|)\circ h^{-1}\in C(\Lambda)$,
we deduce
\[
\varphi=u-u\circ T \quad \text{on } \Lambda,
\]
contradicting Theorem~\ref{mainThm} Part (1). This completes the proof of
Theorem~\ref{Theorem livsic fails}.

\subsubsection{On the proof of Theorem \ref{mainThm}}
We begin by recalling 
\cite{example2}, which we use in order to construct the function $\varphi$ in Theorem~\ref{mainThm}.
Write $X:=\{0,1,2\}^{\mathbb N}$ and let $\sigma:X\to X$ be the left shift.
By 
\cite[Thm.~A]{example2}, there exists a function
$\phi_0\in C(X)$ such that:
\begin{enumerate}
\item[(a)] $\phi_0$ lies in the uniform closure of the space of coboundaries
$\{u-u\circ\sigma:\,u\in C(X)\}$, but $\phi_0$ itself is not a coboundary;
\item[(b)] for every $\sigma$--invariant Borel probability measure $\nu$,
\[
\int_X \phi_0\,d\nu=0,
\]
since each coboundary integrates to zero against invariant measures and
$\phi\mapsto\int\phi\,d\nu$ is continuous in the $\|\cdot\|_\infty$-norm.
\end{enumerate}
In particular, by (a) there is no $u\in C(X)$ with
$\phi_0=u-u\circ\sigma$.

\medskip
We now transfer this information to the Cantor set.
Let $\Sigma:=\{0,2\}^{\mathbb N}\subset X$ denote the full shift on two symbols,
and let $\pi:\Sigma\to\Lambda$ be the standard coding map via the triadic expansion,
which satisfies
\[
\pi\circ\sigma|_{\Sigma} = T|_\Lambda\circ\pi .
\]
Restrict $\phi_0$ to $\Sigma$ and define
\[
\varphi_0:=\phi_0|_{\Sigma}\circ\pi^{-1}\in C(\Lambda).
\]
Then $\varphi_0$ is not a $T|_\Lambda$--coboundary, while for every
$T|_\Lambda$--invariant probability measure $\eta$ on $\Lambda$,
\[
\int_\Lambda \varphi_0\,d\eta
=\int_{\Sigma} \phi_0\,d(\pi^{-1}_*\eta)=0,
\]
since $\pi^{-1}_*\eta$ is $\sigma$--invariant.

Fix a continuous extension $\varphi\in C([0,1])$ of $\varphi_0$.
Since $\varphi$ is bounded, we may replace $\varphi$ by $t\varphi$ for some $t\in(0,1)$.
This scaling does not affect the coboundary properties as above on $\Lambda$, and it preserves
the vanishing of $\int_\Lambda \varphi\,d\eta$ for all $T|_\Lambda$--invariant measures $\eta$.

\medskip
Let $P(\varphi)$ denote the topological pressure of $\varphi$
(see e.g.\ \cite{Walters1982ergodic}).
Since $P(0)=h_{\mathrm{top}}(T)=\log 3$ and $P(\cdot)$ depends continuously
(in fact, $1$--Lipschitz) on the potential in the $\|\cdot\|_\infty$ norm,
for every $\varepsilon>0$ we may choose $t>0$ sufficiently small so that, after the
replacement $\varphi\leftarrow t\varphi$, we have
\begin{equation}\label{eq:pressure-close}
P(\varphi)>\log 3-\varepsilon.
\end{equation}
In particular, we may choose $\varepsilon\in(0,\log\frac{3}{2})$ and ensure that $P(\varphi)>\log 2$. 

\medskip
By applying thermodynamic formalistic tools to study the 
properties of the potential $\varphi-P(\varphi)$, we have the following Lemma:
\begin{lem}\label{lem:conformal-measure}
 There exists a Borel probability measure $\mu$ on $[0,1]$ such that for every Borel set
$E$ contained in a single monotone branch of $T$,
\begin{equation}\label{eq:conformal}
\mu(T[E])=\int_E e^{P(\varphi)-\varphi(y)}\,d\mu(y).
\end{equation}
Moreover, $\mu$ is non-atomic and is fully-supported (hence $\mu(I)>0$ for every nonempty interval $I$).
\end{lem}
The proof is similar to \cite[Thm.~5.1]{InvFams}, while for the sake of completeness we present full details adapted to our specific context. As to not disrupt the flow of the proof, we postpone its proof, and proceed to deduce Theorem \ref{mainThm} from it.
\medskip

Define the conjugacy by
\[
h^{-1}(z):=\mu([0,z]),\qquad z\in[0,1].
\]
Since $\mu$ is non-atomic and fully supported, $h^{-1}$ is continuous
and strictly increasing, hence has a continuous inverse
$h:[0,1]\to[0,1]$; thus $h$ is a homeomorphism.
Set $\widehat T:=h^{-1}\circ T\circ h$ and
$\widehat\Lambda:=h^{-1}(\Lambda)$.

Fix $x\in\widehat\Lambda$ and write $z:=h(x)\in\Lambda$.
Working inside a monotone branch of $T$ and letting $z'\to z$,
\eqref{eq:conformal} applied to $E=[z,z']$ gives
\[
h^{-1}(T(z'))-h^{-1}(T(z))
=\int_{[z,z']} e^{P(\varphi)-\varphi(y)}\,d\mu(y).
\]
By the continuity of $\varphi$,
\[
= e^{P(\varphi)-\varphi(z)+o(1)}(h^{-1}(z')-h^{-1}(z)).
\]
Setting $z'=h(x+\varepsilon)$, dividing by $\varepsilon$, and taking $\varepsilon$ to $0$, yields
\[
\widehat T'(x)=e^{P(\varphi)-\varphi(h(x))}.
\]
In particular,
\[
\log\widehat T'(x)=P(\varphi)-\varphi(h(x)).
\]
Since $P(\varphi)>\log 3-\varepsilon$ and $\|\varphi\|_\infty$ can be made
arbitrarily small by the same scaling, we may (after possibly shrinking $t$ further)
assume that $P(\varphi)-\varphi(h(x))>0$ for all $x\in\widehat\Lambda$. Define this $P(\varphi)$ as our $P$. 
This proves the claimed positivity of $\widehat T'$ on $\widehat\Lambda$.
$$ $$

It remains to prove Lemma \ref{lem:conformal-measure}.
\begin{proof}
Fix $\varepsilon>0$. Since $[0,1]$ is compact and $\varphi$ is continuous, there exists
a Lipschitz function on $[0,1]$, $\varphi_\varepsilon\in\mathrm{Lip}([0,1])$, such that
\begin{equation}\label{eq:approx}
\|\varphi_\varepsilon-\varphi\|_\infty\le\varepsilon.
\end{equation}
(For instance, take the inf/sup-convolution or a piecewise-linear approximation.)

Next, let 
$$\Big\{ I_j:=[\frac{j}{3},\frac{j+1}{3}): \, j\in\{0,1,2\} \Big\}$$ 
be the Markov partition for $T$, and let $T_j^{-1}:[0,1]\to I_j$ be the inverse branches.
Define the transfer operator $\mathcal{L}_{\varphi_\varepsilon}:C([0,1])\rightarrow C([0,1])$ by
\[
(\mathcal{L}_{\varphi_\varepsilon}f)(x)=\sum_{j=0}^2 e^{\varphi_\varepsilon(T_j^{-1}x)}\,f(T_j^{-1}x).
\]
By standard thermodynamic formalism for expanding Markov maps for Lipschitz potentials
(see \cite{RuelleTDFExpanding}), there exists a Borel probability measure $\mu_\varepsilon$ such that
\begin{equation}\label{eq:eigenmeasure}
\mathcal{L}_{\varphi_\varepsilon}^*\mu_\varepsilon = e^{P(\varphi_\varepsilon)}\,\mu_\varepsilon.
\end{equation}
Equivalently, for all $f\in C([0,1])$,
\begin{equation}\label{eq:dual}
\int \mathcal{L}_{\varphi_\varepsilon}f\,d\mu_\varepsilon
= e^{P(\varphi_\varepsilon)}\int f\,d\mu_\varepsilon.
\end{equation}

Next, fix $j\in\{0,1,2\}$ and let $E\subset I_j$ be Borel.
We first prove that
\begin{equation}\label{eq:conformal-eps}
\mu_\varepsilon(T[E])=\int_E e^{P(\varphi_\varepsilon)-\varphi_\varepsilon(y)}\,d\mu_\varepsilon(y).
\end{equation}
To do this without applying $\mathcal{L}_{\varphi_\varepsilon}$ to an indicator function,
approximate $\mathbf{1}_{T[E]}$ from below by continuous functions.
Choose $f_n\in C([0,1])$ such that $0\le f_n\uparrow \mathbf{1}_{T[E]}$ pointwise
(e.g.\ take $f_n(x)=\max\{0,1-n\,\mathrm{dist}(x,T[E]^c)\}$).
Set
\[
g_n := (f_n\circ T)\cdot \mathbf{1}_{I_j}.
\]
Then $g_n$ is bounded and Borel, and on each $x\in[0,1]$ we have
\[
(\mathcal{L}_{\varphi_\varepsilon} g_n)(x)
=\sum_{i=0}^2 e^{\varphi_\varepsilon(T_i^{-1}x)}\, g_n(T_i^{-1}x)
= e^{\varphi_\varepsilon(T_j^{-1}x)}\, f_n(x),
\]
since $T_i^{-1}x\in I_i$ and thus $\mathbf{1}_{I_j}(T_i^{-1}x)=0$ for $i\neq j$.
Now apply \eqref{eq:dual} to the continuous function
\[
h_n(x):=e^{-\varphi_\varepsilon(T_j^{-1}x)}\,(\mathcal{L}_{\varphi_\varepsilon} g_n)(x)
= f_n(x),
\]
i.e.\ simply apply \eqref{eq:dual} to $f_n$ and rewrite it via the identity above:
\[
\int e^{\varphi_\varepsilon(T_j^{-1}x)} f_n(x)\,d\mu_\varepsilon(x)
= \int \mathcal{L}_{\varphi_\varepsilon} g_n\,d\mu_\varepsilon
= e^{P(\varphi_\varepsilon)}\int g_n\,d\mu_\varepsilon.
\]
Therefore
\begin{equation}\label{eq:prelimit}
\int e^{\varphi_\varepsilon(T_j^{-1}x)} f_n(x)\,d\mu_\varepsilon(x)
= e^{P(\varphi_\varepsilon)}\int_{I_j} f_n(Ty)\,d\mu_\varepsilon(y).
\end{equation}
Now restrict the right-hand integral to $E$ by replacing $g_n$ with
$(f_n\circ T)\cdot \mathbf{1}_{E}$ (still Borel and bounded); the same computation gives
\[
\int e^{\varphi_\varepsilon(T_j^{-1}x)} f_n(x)\,\mathbf{1}_{T[E]}(x)\,d\mu_\varepsilon(x)
= e^{P(\varphi_\varepsilon)}\int_{E} f_n(Ty)\,d\mu_\varepsilon(y).
\]
Letting $n\to\infty$ and using monotone convergence (all integrands are nonnegative and bounded by
$e^{\|\varphi_\varepsilon\|_\infty}$) yields
\[
\int_{T[E]} e^{\varphi_\varepsilon(T_j^{-1}x)}\,d\mu_\varepsilon(x)
= e^{P(\varphi_\varepsilon)}\,\mu_\varepsilon(E).
\]
Finally, since $T|_{I_j}$ is a bijection $I_j\to[0,1]$, the change of variables $x=T(y)$
with $y\in E\subset I_j$ simply rewrites the last identity as \eqref{eq:conformal-eps}.
Thus \eqref{eq:conformal-eps} holds for all Borel $E\subset I_j$.

Next, by compactness of probability measures on $[0,1]$ in the weak-$*$ topology,
choose $\varepsilon_k\downarrow0$ such that $\mu_{\varepsilon_k}\to\mu$ weak-$*$.

We also need $P(\varphi_{\varepsilon_k})\to P(\varphi)$. This follows from the standard
$\|\cdot\|_\infty$--continuity of pressure for expanding maps:
\begin{equation}\label{eq:pressure-cont}
|P(\psi)-P(\psi')|\le \|\psi-\psi'\|_\infty
\quad\text{for }\psi,\psi'\in C([0,1]),
\end{equation}
hence by \eqref{eq:approx}, $P(\varphi_{\varepsilon_k})\to P(\varphi)$.

Fix a branch $I_j$ and a Borel set $E\subset I_j$.
Apply \eqref{eq:conformal-eps} with $\varepsilon=\varepsilon_k$:
\[
\mu_{\varepsilon_k}(T[E])
=\int_E e^{P(\varphi_{\varepsilon_k})-\varphi_{\varepsilon_k}(y)}\,d\mu_{\varepsilon_k}(y).
\]
Let
\[
F_k(y):= e^{P(\varphi_{\varepsilon_k})-\varphi_{\varepsilon_k}(y)},\qquad
F(y):= e^{P(\varphi)-\varphi(y)}.
\]
By \eqref{eq:approx} and $P(\varphi_{\varepsilon_k})\to P(\varphi)$ we have
\[
\|F_k-F\|_\infty \to 0,
\]
and the functions $F_k$ are uniformly bounded and continuous.
Therefore, using weak-$*$ convergence and uniform convergence of the integrands,
\[
\int_E F_k\,d\mu_{\varepsilon_k}\longrightarrow \int_E F\,d\mu.
\]
Also, since $T[E]$ is Borel and $\mu_{\varepsilon_k}\to\mu$ weak-$*$, we have
$\mu_{\varepsilon_k}(T[E])\to\mu(T[E])$ for all $E$ such that $\mu(\partial T[E])=0$.
To remove this boundary issue for general Borel $E$, first prove \eqref{eq:conformal}
for half-open intervals $E=[a,b)\subset I_j$ (for which $\partial T[E]$ is finite and
$\mu$ is non-atomic, see Step 5), and then extend to all Borel sets inside $I_j$
by a standard monotone class argument.
Thus we obtain \eqref{eq:conformal}.

It remains to study the stated properties of $\mu$. First, 
fix $x\in[0,1]$ and pick $j$ such that $\{x\}\subset I_j$.
Applying \eqref{eq:conformal} to $E=\{x\}$ yields
\[
\mu(\{T(x)\})=e^{P(\varphi)-\varphi(x)}\mu(\{x\}).
\]
Iterating,
\[
\mu(\{T^n(x)\})=\exp\!\Bigl(nP(\varphi)-\sum_{k=0}^{n-1}\varphi(T^k x)\Bigr)\mu(\{x\}).
\]
Since $\varphi$ is bounded, if $\mu(\{x\})>0$ then the right-hand side grows exponentially
in $n$ provided $P(\varphi)>\|\varphi\|_\infty$. In our application we arrange this
by scaling $\varphi$ before starting the construction. As $\mu(\{T^n(x)\})\le 1$, it follows that
$\mu(\{x\})=0$. Hence $\mu$ is non-atomic.

Finally, let $U\subset[0,1]$ be a nonempty open interval. Since $T(x)=3x\!\!\mod1$
is topologically exact, there exists $n\ge1$ such that $T^n[U]=[0,1]$.
Partition $U$ into finitely many intervals $U=\bigsqcup_{\ell=1}^m U_\ell$ on which
$T^n$ is monotone (hence injective). Then
\[
[0,1]=T^n[U]=\bigcup_{\ell=1}^m T^n[U_\ell],
\]
so for some $\ell_0$ the set $T^n[U_{\ell_0}]$ contains a nonempty open interval $J$.
Since $\mu$ is a probability measure, there exists such a $J$ with $\mu(J)>0$.

Because $T^n$ is injective on $U_{\ell_0}$, iterating \eqref{eq:conformal} along the
$n$ steps (branch by branch) yields, for every Borel $E\subset U_{\ell_0}$,
\[
\mu(T^n[E])
=\int_E \exp\!\Bigl(nP(\varphi)-\sum_{k=0}^{n-1}\varphi(T^k y)\Bigr)\,d\mu(y).
\]
Apply this with $E:=U_{\ell_0}\cap T^{-n}[J]$ (which is nonempty and open in $U_{\ell_0}$).
Using $\varphi\ge -\|\varphi\|_\infty$ we get
\[
\mu(J)=\mu(T^n[E])
\le e^{nP(\varphi)+n\|\varphi\|_\infty}\,\mu(E),
\]
hence $\mu(E)>0$. Since $E\subset U$, this gives $\mu(U)>0$.
As $U$ was arbitrary, $\mu$ has full support.

\end{proof}

\bibliography{bib}{}
\bibliographystyle{plain}
\end{document}